\documentclass[reqno,centertags, 12pt]{amsart}
\usepackage{comment,graphicx,url}
\usepackage{color}

\definecolor{purple}{rgb}{0.65, 0, 0.9}
\definecolor{orange}{rgb}{1,.5,0}
\definecolor{gray}{rgb}{0.7,.7,0.7}

\usepackage{enumitem}

\usepackage{graphics}
\usepackage{amssymb,amsmath,amsfonts}
-\marginparwidth \oddsidemargin -9mm \evensidemargin \oddsidemargin
\usepackage{latexsym}
\advance\hoffset by 5mm

\def\@abssec#1{\vspace{.1in}\footnotesize \parindent .2in
{\bf #1. }\ignorespaces}
\newtheorem{theorem}{Theorem}[section]

\newtheorem{lemma}[theorem]{Lemma}
\newtheorem{proposition}[theorem]{Proposition}

\newtheorem{remark}[theorem]{Remark}

\def \R {\mathbb R}
\def \T {\mathbb T}
\def \N {\mathbb N}

\def \Z {\mathbb Z}

\newcommand{\be}{\mathbf e}

\allowdisplaybreaks \numberwithin{equation}{section}

\renewcommand{\be}{\begin{equation}}
\newcommand{\ee}{\end{equation}}

\begin{document}

\title[Small scale formation for the 2D Boussinesq equation]{Small scale formation for the 2D Boussinesq equation}
\author{Alexander Kiselev}
\thanks{Department of
Mathematics, Duke University, 120 Science Dr., Durham NC 27708, USA;
email: kiselev@math.duke.edu}

\author{Jaemin Park}
\thanks{
Department of Mathematics and Computer Science, University of Basel,  Spiegelgasse 1, 4051 Basel, Switzerland. email: jaemin.park@unibas.ch}

\author{Yao Yao}
\thanks{
Department of Mathematics, National University of Singapore,  10 Lower Kent Ridge Road, Singapore 119076. email: yaoyao@nus.edu.sg}

\begin{abstract}We study the 2D incompressible Boussinesq equation without thermal diffusion, and aim to construct rigorous examples of small scale formations as time goes to infinity. In the viscous case, we construct examples of global smooth solutions satisfying $\sup_{\tau\in[0,t]} \|\nabla \rho(\tau)\|_{L^2}\gtrsim t^\alpha$ for some $\alpha>0$. For the inviscid equation  in the strip, we construct examples satisfying $\|\omega(t)\|_{L^\infty}\gtrsim t^3$ and $\sup_{\tau\in[0,t]} \|\nabla \rho(\tau)\|_{L^\infty} \gtrsim t^2$ during the existence of a smooth solution. These growth results hold for a broad class of initial data, where we only require certain symmetry and sign conditions. 
As an application, we also construct solutions to the 3D axisymmetric Euler equation whose velocity has infinite-in-time growth.\end{abstract}

\maketitle

\section{Introduction}
The incompressible Boussinesq equations describe the motion of incompressible fluid under the influence of gravitational forces \cite{gill1982atmosphere, majda2003introduction, pedlosky1987geophysical}. Let us denote by $\rho(x,t)$ the density of the fluid (it can also represent the temperature, depending on the physical context), and $u(x,t)$ the velocity field. Throughout this paper, we consider the  2D incompressible Boussinesq equation in the absence of density/thermal diffusivity:
\begin{equation}\label{Boussinesq}
\begin{split} &\rho_t + u \cdot \nabla \rho = 0,  \\
&u_t + u\cdot \nabla u = -\nabla p - \rho e_2 + \nu \Delta u,  \qquad x\in\Omega, t>0,\\
&\nabla \cdot u = 0, \\
\end{split}
\end{equation}
where the initial condition is $u(\cdot,0)=u_0$ and $\rho(\cdot,0)=\rho_0$.
Here $e_2 := (0,1)^T$, and $\nu \geq 0$ is the viscosity coefficient. We assume the spatial domain $\Omega$ is one of the following: the whole space $\R^2$, the torus $\mathbb{T}^2:= \left(- \pi ,\pi \right]^2$, or the strip $\mathbb{T}\times [0,\pi]$ that is periodic in $x_1$. When $\Omega$ is the strip, we impose the  no-slip boundary condition $u|_{\partial \Omega}=0$ if $\nu>0$; and the no-flow boundary condition $u\cdot n|_{\partial \Omega}=0$ if $\nu=0$.

In the past decade, much progress has been made on the analysis of \eqref{Boussinesq} in both the viscous case $\nu>0$ and inviscid case $\nu=0$. Below we briefly review the relevant literature, and state our main results in each case. 

\subsection{The viscous case $\nu>0$.} If the equation for $\rho$ has an additional thermal diffusion term $\kappa\Delta\rho$, global regularity of solutions is well-known (see e.g. \cite{temam2012infinite}) and follows  from the classical methods for Navier--Stokes equations. In the absence of thermal diffusion, the first global-in-time regularity results were obtained by Hou--Li \cite{HouLi2005} in the space $(u,\rho)\in H^m(\mathbb{R}^2)\times H^{m-1}(\mathbb{R}^2)$ for $m\geq 3$, and Chae \cite{Chae2006} in the space $H^m(\mathbb{R}^2)\times H^m(\mathbb{R}^2)$ for $m\geq 3$. When $\Omega\subset\mathbb{R}^2$ is a bounded domain, Lai--Pan--Zhao \cite{LaiPanZhao2011} proved global well-posedness of solutions in $H^3(\Omega)\times H^3(\Omega)$ with no-slip boundary condition, and showed that the kinetic energy is uniformly bounded in time. The function space was improved by Hu--Kukavica--Ziane \cite{HKZ2013} to $(u,\rho)\in H^m(\Omega)\times H^{m-1}(\Omega)$ for $m\geq 2$, where $\Omega$ is either a bounded domain or $\mathbb{R}^2$, $\mathbb{T}^2$. In spaces with lower regularity, global well-posedness of weak solutions was obtained by Abidi--Hmidi \cite{MR2290277},  Hmidi--Keraani \cite{MR2305876}, Danchin--Paicu \cite{DP2011}, and Larios--Lunasin--Titi \cite{LLT2013}. For the temperature patch problem, Gancedo--Garc{\'\i}a-Ju{\'a}rez \cite{gancedo2017global,gancedo2020regularity} proved global regularity in 2D, and local regularity in 3D.

Regarding upper bounds of the global-in-time solutions, for a bounded domain, Ju~\cite{Ju2017} obtained that
$\|\rho\|_{H^1(\Omega)} \lesssim e^{Ct^2}$. The $e^{Ct^2}$ bound was improved into an exponential bound $e^{Ct}$ in Kukavica--Wang \cite{KW2020} for $\Omega=\mathbb{T}^2$ or a bounded domain, and a super-exponential bound $e^{Ct^{(1+\beta)}}$ for some constant $\beta\approx 0.29$ for $\Omega=\mathbb{R}^2$. When $\Omega=\mathbb{T}^2$, they also obtained the uniform-in-time bound $\|u\|_{W^{2,p}(\mathbb{T}^2)}\leq C(p)$ for all $p\in[2,\infty)$. In a recent work by Kukavica--Massatt--Ziane \cite{kukavica2021asymptotic}, when $\Omega$ is a bounded domain, the upper bound of the norm of $\rho$ has been improved to $\|\rho\|_{H^2(\Omega)}\leq C_\epsilon e^{\epsilon t}$ for all $\epsilon>0$, and they also showed $\|u\|_{H^3}\leq C_\epsilon e^{\epsilon t}$ for all $\epsilon>0$.

We would like to point out that all these results deal with \emph{upper bounds} of solutions, and it is a natural question whether certain norms of solutions \emph{can} actually grow to infinity as $t\to\infty$. When $\nu>0$ and $\Omega=\mathbb{R}^2$, Brandolese--Schonbek \cite{BS2012} proved that when the initial data $\rho_0$ does not have mean zero, $\|u(t)\|_{L^2(\mathbb{R}^2)}$ must grow to infinity like $ (1+t)^{1/4}$. Here the growth mechanism is due to potential energy converting into kinetic energy, and does not necessarily imply growth in higher derivatives of $u$ or $\rho$. To the best of our knowledge, there has been no example in literature showing that $\|\rho(t)\|_{\dot{H}^m}$ or $\|u(t)\|_{\dot{H}^m}$ can actually grow to infinity as $t\to\infty$ for some $m\geq 1$. The goal of this paper is exactly to construct such examples in $\mathbb{R}^2$ and $\mathbb{T}^2$ where $\|\rho(t)\|_{\dot{H}^m}\to\infty$ as $t\to\infty$ for all $m\geq 1$. Since $\|\rho(t)\|_{L^2}$ is preserved in time, growth of $\|\rho(t)\|_{\dot{H}^m}$ implies that $\rho$ has some small scale formation as $t\to\infty$.

\medskip
 In the viscous case, we set the spatial domain to be either $\mathbb{R}^2$ or $\mathbb{T}^2$, and assume that the initial data $(\rho_0,u_0)$ satisfies the following assumptions (here we denote $u_0 = (u_{01}, u_{02})^T$). See Figure~\ref{fig1} for an illustration of the assumptions on $\rho_0$.
\begin{enumerate}[label=\textbf{(A\arabic*)}]
\item \label{assumption1_rho} $\rho_0, u_0 \in C^\infty(\Omega)$. If $\Omega = \mathbb{R}^2$, assume in addition that $\rho_0, u_0 \in C^\infty_c(\R^2)$.
\item\label{assumption2_rho}$\rho_0$  and $u_{02}$ are odd in $x_2$, and $u_{01}$ is even in $x_2$. If $\Omega = \mathbb{T}^2$, assume in addition that $\rho_0$  and $u_{02}$ are even in $x_1$,  $u_{01}$ is odd in $x_1$, and $\rho_0=0$ on the $x_2$-axis.\footnote{Note that if the $\rho_0=0$ on the $x_2$-axis   assumption is removed, the initial data would include some steady states with horizontally stratified density, which clearly would not lead to any growth.}

 \item\label{assumption3_rho} $\rho_0$ is not identically zero, and $\rho_0 \ge 0$ for $x_2\ge 0$.
\end{enumerate}

\begin{figure}
\hspace{0.3cm}\includegraphics[scale=1]{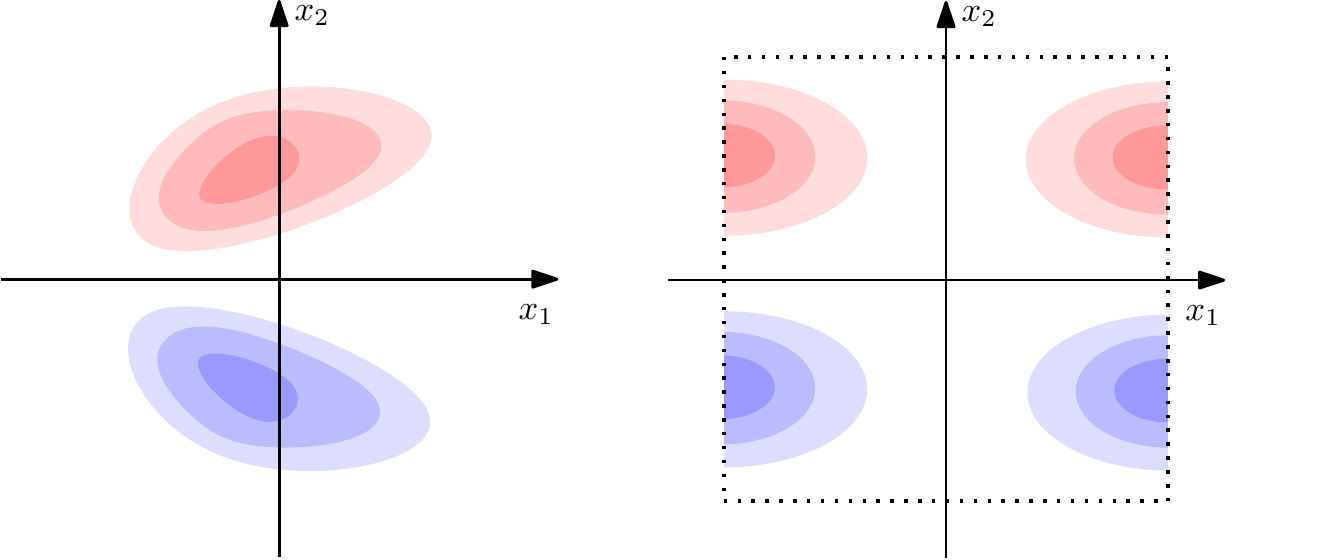}
\caption{Illustration of the symmetry and sign assumptions on $\rho_0$ in the plane (left) and torus $\mathbb{T}^2$ (right) for the viscous Boussinesq equation. Here red color denotes positive $\rho_0$, and blue color denotes negative $\rho_0$. \label{fig1}}
\end{figure}

As we will see in Section~\ref{sec21}, under these assumptions, both the potential energy $E_P(t) := \int_\Omega \rho(x,t) x_2 dx$ and kinetic energy $E_K(t)=\frac{1}{2}\|u(t)\|_{L^2(\Omega)}^2$ of the solution remain bounded for all times, and the total energy is decreasing in time. We prove that for all $s\geq 1$, the Sobolev norm $\|\rho(t)\|_{\dot{H}^s}$ grows to infinity at least algebraically in $t$:

\begin{theorem}\label{infinite_growth}
Assume $\nu>0$, and let $\Omega = \R^2$ or $\mathbb{T}^2$. For any initial data $(\rho_0,u_0)$ satisfying \textbf{\textup{(A1)}}--\textbf{\textup{(A3)}}, the global-in-time smooth solution $(\rho,u)$ to \eqref{Boussinesq} satisfies the following:
\begin{itemize}
\item If $\Omega = \mathbb{R}^2$, we have 
\begin{equation}
\limsup_{t\to \infty}t^{-\tfrac{s}{10}}\| \rho(t)\|_{\dot{H}^s(\Omega)} = +\infty \quad \text{ for all }s\ge 1;\label{growth1}
\end{equation}
\item  If $\Omega = \mathbb{T}^2$, we have
\begin{equation}
\limsup_{t\to \infty}t^{-\tfrac{s(2s -1)}{8s -2}}\| \rho(t)\|_{\dot{H}^s(\Omega)}= +\infty  \quad \text{ for all }s\ge 1.\label{growth2}
\end{equation}
\end{itemize}
\end{theorem}

\begin{remark} It is a natural question whether these growth rates are sharp. While the powers are likely non-sharp, we point out that $\|\rho(t)\|_{H^1}$ cannot have exponential growth under the assumptions \textbf{\textup{(A1)}}--\textbf{\textup{(A3)}}. Namely, following the arguments similar to Kukavica--Wang \cite{KW2020}, we show in Proposition~\ref{prop_subexp} that under the assumptions \textbf{\textup{(A1)}}--\textbf{\textup{(A3)}}, $\|\rho(t)\|_{H^1}$ has a refined sub-exponential upper bound
\[
\|\rho(t)\|_{H^1(\Omega)} \lesssim \exp (Ct^{\alpha})\quad\text{ for all }t>0
\]
for some constant $\alpha\in(0,1)$. Therefore in this setting, the fastest possible growth rate of $\| \rho(t)\|_{H^1(\Omega)}$ is somewhere between algebraic and sub-exponential.

\end{remark}

The proof of Theorem~\ref{infinite_growth} is motivated by a recent result on small scale formation in solutions to incompressible porous media (IPM) equation by the first and third author \cite{KY}. The main idea there was to use the monotonicity of the potential energy $E_P(t) = \int \rho(x,t) x_2 dx$: on the one hand, for solutions with certain symmetries, $E_P(t)$ is bounded below with $E_P'(t)=-\|\partial_1\rho(t)\|_{\dot{H}^{-1}}^2$, thus the integral $\int_0^\infty \|\partial_1\rho(t)\|_{\dot{H}^{-1}}^2 dt$ is finite; on the other hand, under certain symmetries, one can show that $\|\partial_1\rho(t)\|_{\dot{H}^{-1}}^2$ can only be small if $\|\rho(t)\|_{H^s}\gg 1$ for some $s>0$, leading to growth of $\rho$ in Sobolev norms.

The IPM and Boussinesq equation are related in the sense that in both equations, the density $\rho$ is transported by an incompressible $u$, where 
$
u = -\nabla p - \rho e_2
$ in IPM,
whereas
$
\frac{Du}{Dt} = -\nabla p - \rho e_2 + \nu \Delta u
$ in Boussinesq equations. Since the velocity in Boussinesq equation has one more time derivative than IPM,  we formally expect that $E_P''(t)$ should be related to $-\|\partial_1\rho(t)\|_{\dot{H}^{-1}}^2$. While this turns out to be true, the situation is more delicate for the Boussinesq equation because $E_P''(t)$ also contains other terms coming from the pressure and viscosity terms. By carefully controlling these additional terms, we prove that if $\|\rho(t)\|_{H^s}$ grows too slowly for $s\geq 1$, $E_P'(t)$ would become unbounded below, contradicting the uniform-in-time bound of energy.

\subsection{ The inviscid case $\nu=0$.} For the inviscid Boussinesq equations in 2D, it is well-known that the system \eqref{Boussinesq} can be rewritten into an equivalent system for the density $\rho$ and the vorticity $\omega= \partial_1 u_2 -\partial_2 u_1$:
 \begin{equation}\label{Boussinesqw}
\begin{split}
&\rho_t + u \cdot \nabla \rho = 0, \\
&\omega_t + u\cdot \nabla \omega = -\partial_1 \rho, \\
\end{split}
\end{equation}
where the velocity $u$ can be recovered from the vorticity $\omega$ from the Biot--Savart law $u = \nabla^\perp (-\Delta)^{-1}\omega$.
While
local well-posedness results are available in a variety of functional spaces for $\Omega=\mathbb{R}^2, \mathbb{T}^2$ or a bounded domain \cite{CKN1999,ChaeNam1997,Danchin2013},  whether smooth initial data in $\mathbb{T}^2$ or $\mathbb{R}^2$ with finite energy can develop a finite-time singularity is an outstanding open question in fluid dynamics. Note that smooth, infinite-energy initial data can lead to a finite time blow-up, as shown by Sarria--Wu \cite{SW2015}.

In the presence of boundary, there have been many exciting developments regarding finite-time singularity formation of solutions in the past few years. 
Luo--Hou \cite{HL2014} provided numerical evidence for finite time blow up in smooth solutions of the 3D axi-symmetric Euler 
equation in a cylinder.  
When the domain has a corner, Elgindi--Jeong \cite{elgindi2020finite} proved that blow-up can happen for inviscid Boussinesq equation with smooth initial data.
When $\Omega=\mathbb{R}^2_+$ is the upper-half-plane, Chen--Hou \cite{CH2021} proved that solutions with $C^{1,\alpha}$ velocity and density can have a nearly self-similar finite-time blowup. Recently, for smooth initial data, Wang--Lai--G\'omez-Serrano--Buckmaster \cite{wang2022self} used physics-informed neural networks to construct an approximate self-similar blow-up solution numerically. In a very recent preprint, Chen--Hou \cite{ChenHou2022} put forward an argument combining impressive analytical tools and computer assisted estimates 
to show that smooth initial data can lead to a stable nearly self-similar blowup.

 Note that the inviscid Boussinesq equation \eqref{Boussinesqw} becomes the 2D Euler equation when $\rho\equiv 0$, where it is well-known that $\|\nabla\omega(t)\|_{L^\infty}$ can have infinite-in-time growth \cite{MR2461825, MR3293735,MR3245016, MR1139875, MR3276599}. Therefore we will only focus on proving infinite-in-time growth of either $\nabla \rho$ (since $\rho$ itself is preserved along the trajectory, one can at most obtain growth results for $\nabla \rho$), or $L^p$ norms of $\omega$ itself not involving any derivatives (where such growth is not possible for 2D Euler since the $\|\omega\|_{L^p}$ is preserved in time).

Our first result is set up in the periodic domain $\Omega=\mathbb{T}^2$. We show that for all smooth initial data $(\rho_0,\omega_0)$ in $\mathbb{T}^2$ under some symmetry assumptions, as long as $\rho_0$ takes values of different sign along the two line segments $\{0\}\times[0,\pi]$ and $\{\pi\}\times[0,\pi]$ (see the left figure of Figure~\ref{fig2} for an illustration), $\|\nabla\rho(t)\|_{L^\infty}$ must grow to infinity at least algebraically in time for all time during the existence of a smooth solution.

\begin{theorem}\label{invT2}
 Let $\rho_0\in C^\infty(\mathbb{T}^2)$ be odd in $x_2$ and even in $x_1$, and $\omega_0\in C^\infty(\mathbb{T}^2)$ be odd in both $x_1$ and $x_2$. Assume $\rho_0\geq 0$ on $\{0\}\times[0,\pi]$ with $k_0:= \sup_{x_2\in [0,\pi]}\rho_0(0,x_2)>0$, and $\rho_0\leq 0$ on $\{\pi\}\times[0,\pi]$.  Then there exists some constant $ c(\rho_0,\omega_0)>0$, such that the corresponding solution $(\rho,\omega)$ to \eqref{Boussinesqw} satisfies
\begin{equation}\label{rholbiT2}
\sup_{\tau\in[0,t]}\|\nabla \rho (\tau)\|_{L^\infty(\mathbb{T}^2)} > c(\rho_0,\omega_0)  t^{1/2} \quad\text{ for all }t\in[0,T),
\end{equation}  
where  $T$ is the lifespan of the smooth solution $(\rho,\omega)$.

\end{theorem}

\begin{figure}[hbp]
\hspace{0.3cm}\includegraphics[scale=0.9]{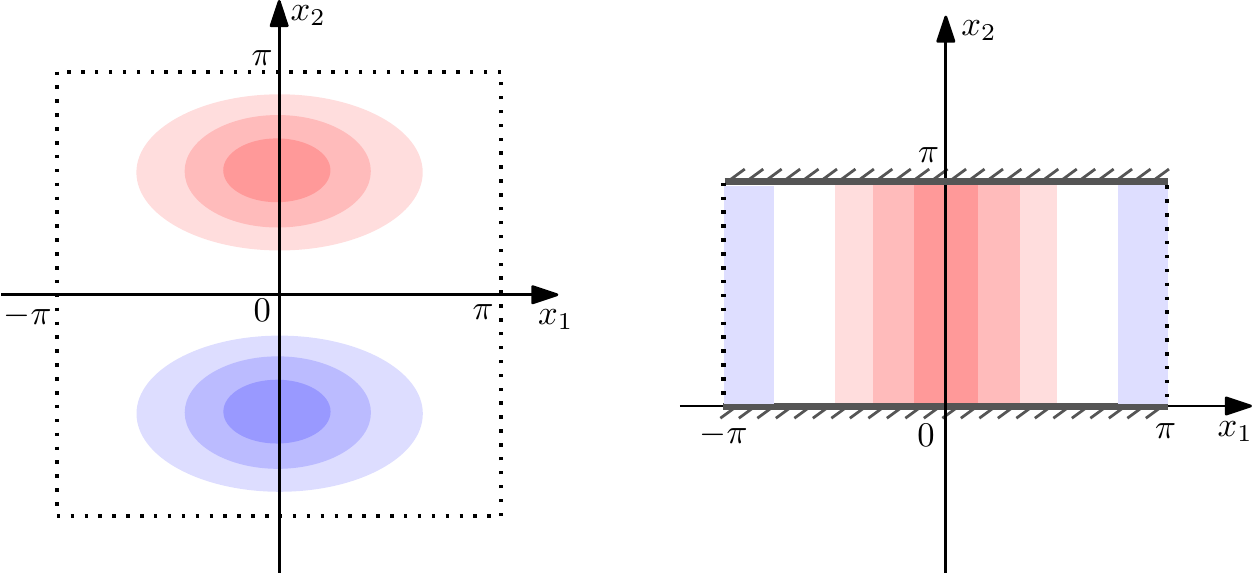}
\caption{Illustration of  the symmetry and sign assumptions on $\rho_0$ in the torus $\mathbb{T}^2$ (left) and the strip $\mathbb{T}\times[0,\pi]$ (right) for the inviscid Boussinesq equation. Here red color denotes positive $\rho_0$, and blue color denotes negative $\rho_0$. \label{fig2}}
\end{figure}

Next we consider the inviscid Boussinesq equation in the strip $\mathbb{T}\times[0,\pi]$. Here the presence of boundary allows us to obtain a faster growth rate in $\|\nabla\rho(t)\|_{L^\infty}$: we prove that the growth is at least like $t^2$ in the strip (as compared to $t^{1/2}$ in Theorem~\ref{invT2}). We are also able to obtain a superlinear lower bound for $\|\omega(t)\|_{L^p}$ (for $p=\infty$ it grows like $t^3$) and a linear lower bound for $\|u(t)\|_{L^\infty}$.
Although these algebraic lower bounds are far from finite-time blow up, they hold for a broad class of initial data: no assumption on $\omega_0$ is needed other than being odd in $x_1$, and $\rho_0$ only needs to be even in $x_1$ and satisfy some sign conditions along two line segments (see the right figure of Figure~\ref{fig2} for an illustration). The proofs are soft but might provide an insight into the behavior of smooth solutions during their lifespan.

\begin{theorem}\label{invB}
Let $\Omega = \T \times [0,\pi].$ Let $\rho_0\in C^\infty(\Omega)$ be even in $x_1$, and $\omega_0\in C^\infty(\Omega)$ be odd in $x_1$. Assume that there exists $k_0>0$ such that $\rho_0\geq k_0>0$ on $\{0\}\times[0,\pi]$, and $\rho_0\leq 0$ on $\{\pi\}\times[0,\pi]$. Then there exist some constants $T_0(\rho_0,\omega_0)\geq 0$ and $c(\rho_0,\omega_0)>0$, such that the corresponding solution $(\rho,\omega)$ to \eqref{Boussinesqw} satisfies
\begin{equation}\label{omlbiB}
\|\omega(t)\|_{L^p(\Omega)} \geq c t^{3-\frac2p} \quad\text{ for all }p\in [1,\infty], t\in[T_0,T),
\end{equation}
\begin{equation}\label{u_growth}
\|u(t)\|_{L^\infty(\Omega)} \geq c t \quad\text{ for all }t\in[T_0,T),
\end{equation}
and
\begin{equation}\label{rholbiB}
\sup_{\tau\in[0,t]}\|\nabla\rho (\tau)\|_{L^\infty(\Omega)} > c t^2 \quad\text{ for all } t\in[0,T),
\end{equation}
where  $T$ is the lifespan of the smooth solution $(\rho,\omega)$.
In particular, if $\int_{[0,\pi]\times[0,\pi]}\omega_0 dx \geq 0$, then $T_0=0$ in all the estimates above. 
\end{theorem}

\begin{remark}
 In the estimates for $\|\omega(t)\|_{L^p(\Omega)}$ and $\|u(t)\|_{L^\infty(\Omega)}$ above, it is necessary to have a ``waiting time'' $T_0$ depending on the initial data. This is because for any $t_1>0$, there exists some initial data satisfying the assumption of Theorem~\ref{invB} with $\omega(\cdot,t_1)\equiv 0$. (To see this, one can start with $\omega(\cdot,t_1)\equiv 0$ and go backwards in time). That being said, it can be easily seen from the proof that if $\int_{[0,\pi]\times[0,\pi]}\omega_0 dx \geq 0$, no waiting time is needed.
\end{remark}

\begin{remark}
 If the symmetry assumptions on $\rho_0$ and $\omega_0$ are dropped, we still have $\|\omega(t)\|_{L^1(\Omega)}\gtrsim t$ for $t\gg 1$. This infinite-in-time growth implies that given any steady state $\omega_s$ for the 2D Euler equation on the strip, we have $(0,\omega_s)$ is a nonlinearly unstable steady state for the inviscid Boussinesq equation. See Remark~\ref{rmk_stable} for more discussions.
\end{remark}

For both Theorems~\ref{invT2} and \ref{invB},  the proof is based on an interplay between various monotone and conservative quantities. Under the symmetry assumptions, one can easily check that the sign assumptions  $\rho \geq 0$ on $\{0\}\times[0,\pi]$ and $\rho\leq 0$ on $\{\pi\}\times[0,\pi]$ remain true for all times. This allows us to make the elementary but important observation that the vorticity integral $\int_{[0,\pi]\times[0,\pi]}\omega(x,t)dx$ is monotone increasing for all times. More precisely, for the strip the growth is linear for all times during the existence of a smooth solution, whereas in $\mathbb{T}^2$ we relate the growth with $\|\nabla\rho(t)\|_{L^\infty}$. Another key ingredient is the relation between vorticity integral and kinetic energy: since the kinetic energy has a uniform-in-time bound, we prove that if the vorticity integral is large, the $L^p$ norm of vorticity must be much larger. For a strip, this allows us to upgrade the linear growth of $\|\omega(t)\|_{L^1}$ into superlinear growth for  $\|\omega(t)\|_{L^p}$ for $p\in(1,\infty]$.

\subsection{Infinite-in-time growth for the 3D axisymmetric Euler equation.}
The question whether incompressible Euler equation in $\mathbb{R}^3$ can have a finite-time blow-up from smooth initial data of finite energy is an outstanding open problem in
nonlinear PDE and fluid dynamics. As we mentioned earlier, for the 3D axisymmetric Euler equation, when the equation is set up in a cylinder with boundary, Luo--Hou \cite{HL2014} gave convincing numerical evidence that smooth initial data can lead to a finite-time singularity formation on the boundary. Recent numerical evidence by Hou--Huang \cite{hou2021potential, hou2022potential} and Hou \cite{hou2022b} suggests that the blow-up can also happen in the interior of domain, but apparently not in self-similar fashion. The first rigorous  blow-up result for finite energy solutions was established in domains with corners by Elgindi--Jeong \cite{elgindi2019finite}.
For initial data in $C^{1,\alpha}$ in $\mathbb{R}^3$, Elgindi showed \cite{Elgindi2021} that such initial data can lead to a self-similar blow-up. Very recently, using the connection between 3D axisymmetric Euler and Boussinesq equations, 
Chen--Hou \cite{ChenHou2022} set up a computer assisted argument that smooth solutions to 3D axi-symmetric Euler equation can form a stable nearly self-similar blowup. 
The singularity formation happens for initial data in a small neighborhood of a profile that is selected carefully with computer assistance.   

In addition to the blow-up v.s. global-in-time regularity question, it is also interesting to investigate whether Sobolev norms  of solutions to the 3D Euler equation can have infinite-in-time growth for broader classes of initial data. Choi--Jeong \cite{choi2021filamentation} constructed smooth compactly supported initial data in $\mathbb{R}^3$ with $\|\nabla^2\omega(t)\|_{L^\infty}$ growing algebraically for all times, and $\|\omega(t)\|_{L^\infty}$ growing exponentially for finite (but arbitrarily long) time. It is also well-known that the ``two-and-a-half dimensional'' solutions (i.e. where $u$ only depends on $x,y$, not $z$) can lead to infinite-in-time linear growth of $\omega$; see Bardos--Titi \cite[Remark 3.1]{bardos2007euler} for example. See the excellent survey by Drivas--Elgindi \cite{drivas2022singularity} for more results on growth and singularity formation for 2D and 3D Euler equations.

 It is well-known that away from the axis of symmetry, the 3D axisymmetric Euler equation is closely related to the inviscid 2D Boussinesq equations  (see \cite[Section 5.4.1]{Majda-Bertozzi:vorticity-incompressible-flow}). To see this connection, recall that the 3D axi-symmetric Euler equation can be reduced to the system 
\begin{equation}\label{3dE}
\begin{split}
&D_t (r u^\theta) =0, \\
&D_t \left( \frac{\omega^\theta}{r} \right) = \frac{ \partial_z (r u^\theta )^2 }{r^4},
\end{split}
\end{equation} 
where $u^\theta$ and $\omega^\theta$ only depend on $r,z,t$, and $D_t := \partial_t + u^r \partial_r + u^z \partial_z$ is the material derivative. Heuristically speaking, $ru^\theta$ plays the role of $\rho$ in Boussinesq equation, whereas $\frac{\omega^\theta}{r}$ plays the role of $\omega$ in the Boussinesq equation. Here $(u^r, u^z)$ can be recovered from $\omega^\theta/r$ by the Biot-Savart law
\begin{equation}\label{bs_euler}
(u^r, u^z) = \frac1r \big( -\!\partial_z \psi, \,\partial_r \psi \big), \quad\text{ where }
-\frac1r \partial_r \left( \frac1r \partial_r \psi \right) - \frac{1}{r^2} \partial_z^2 \psi = \frac{\omega^\theta}{r}.   
\end{equation}

We note that the analog of Theorem \ref{invB} holds for the 3D axi-symmetric Euler equation. We set the spatial domain to be a (not rotating) Taylor--Couette tank 
\begin{equation}\label{cylinder}
\Omega = \{ (r,\theta,z): r \in [\pi,2\pi] , \theta\in \T ,z\in \T \}
\end{equation} with no-penetration boundary condition at $r=\pi,2\pi$ and periodic boundary conditions in $z$. Our assumptions and results are as follows:

\begin{figure}[htbp]
\hspace{0.3cm}\includegraphics[scale=0.95]{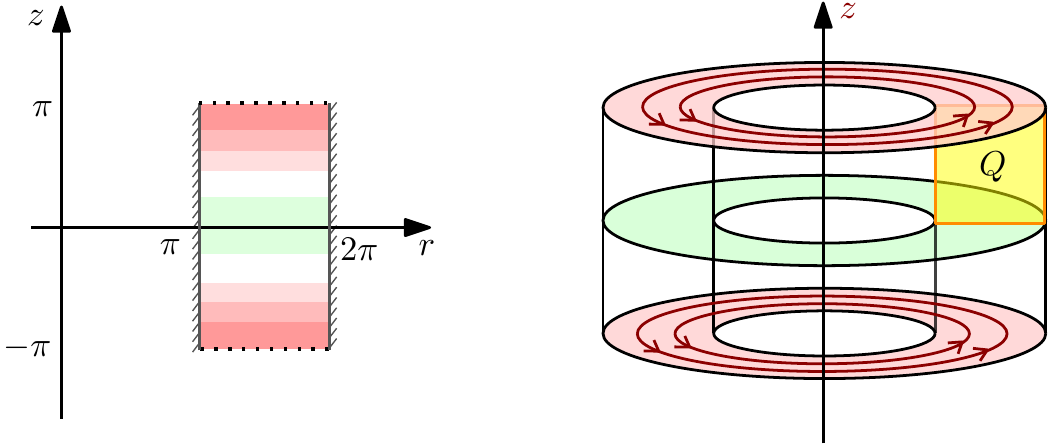}
\caption{Illustration of  the domain and assumptions on $u^\theta_0$ for the 3D axisymmetric Euler equation. The left figure illustrates $u^\theta_0$ on the $rz$ plane, and the right figure shows the 3D setting. Here red color denotes positive $u^\theta_0$ (and deeper color means larger magnitude), and green color denotes $u^\theta_0$ with a smaller magnitude (whose sign can be positive or negative). With such initial data, we will show that the ``secondary flow'' within the yellow square $Q$ grows to infinity as $t\to\infty$. \label{fig3}}
\end{figure}

\begin{theorem}\label{3dEth}
Consider the 3D axisymmetric Euler equation \eqref{3dE}--\eqref{bs_euler} set on the domain $\Omega$ in \eqref{cylinder}.  Let $u^\theta_0\in C^\infty(\Omega)$ be even in $z$, and $\omega^\theta_0\in C^\infty(\Omega)$ be odd in $z$. Assume that there exists $k_0>0$ such that $u^\theta_0\geq k_0>0$ on $z=\pi$, and $|u^\theta_0|\leq \frac{1}{8}k_0$ on $z=0$.
 Then there exist some constants $T_0(u_0)\geq 0$ and $c(u_0)>0 $, such that the corresponding solution satisfies
\begin{equation}\label{omlbiE}
\|\omega^\theta(t)\|_{L^p(\Omega)} \geq c t^{3-\frac{2}{p}} \quad\text{ for all }p\in[1,\infty], t\in[T_0,T)
\end{equation}
and
\begin{equation}\label{euler_u_bd}
\|u(t)\|_{L^\infty(\Omega)}\geq c t \quad\text{ for all } t\in[T_0,T),
\end{equation}
where  $T$ is the lifespan of the smooth solution. In particular, if $\int_0^\pi \int_{\pi}^{2\pi} \omega_0^\theta drdz \geq 0$, then $T_0=0$ in both estimates above. 
\end{theorem}

See Figure~\ref{fig3} for an illustration of the domain and initial data. Note that our setting is almost the same as the Hou--Luo scenario \cite{HL2014}, except that we replace the cylinder by an annular cylinder. While our growth estimates are far from a finite-time blow-up,
they hold for a broad class of initial data: in addition to some symmetry assumptions on $u^\theta_0$ and $\omega^\theta_0$, all we need is $u^\theta_0$ being uniformly positive on $z=\pi$, and having small magnitude on $z=0$. The proof is a simple argument analogous to Theorem~\ref{invB} for Boussinesq equations, where the key idea is the interplay between the monotonicity of an vorticity integral and the boundedness of kinetic energy.

\subsection*{Acknowledgements}
AK  was partially supported the NSF-DMS grant 2006372.
YY was partially supported by the NUS startup grant A-0008382-00-00 and MOE Tier 1 grant A-0008491-00-00.

\section{Small scale formation for viscous Boussinesq equation}

\label{sec_prelim}

 In this section, we aim to prove Theorem~\ref{infinite_growth}. To begin with, we discuss some properties on the solution $(\rho,u)$ when the initial data satisfies  \textbf{(A1)}--\textbf{(A3)}. Under the assumption \textbf{(A1)}, it is well-known  that $\rho(\cdot, t)$ and $u(\cdot,t)$ remain in $C^\infty(\Omega)$. And if $\Omega = \mathbb{R}^2$, we have $\rho(\cdot, t) \in C^\infty_c(\R^2)$, and $u(\cdot, t)\in H^k(\R^2)$ for all $k\in \N$ and $t\geq 0$ (see e.g. \cite{HouLi2005,Chae2006}).

    Note that the symmetry in  \textbf{(A2)} holds true for all times thanks to the uniqueness of solutions. If $\Omega=\T^2$, the additional symmetry in $x_1$ leads to $u_1(\cdot,t)=0$ on the $x_2$-axis for all times, thus $\rho(0,x_2,t)=0$ for all $x_2 \in \T$ and $t\geq 0$.

 The symmetry in $x_2$ in  \textbf{(A2)} also gives $u_2(\cdot,t)=0$  on the $x_1$-axis for all times, and combining it with \textbf{(A3)} gives  $\rho(x,t) \ge0$ for $x_2 \ge 0$ and all $t\geq 0$.

 We also point out that due to the incompressibility of $u$, all $L^p$ norms of $\rho$ are conserved in time, that is,
 \begin{equation}\label{lp_conservation}
 \rVert \rho(t,\cdot) \rVert_{L^p(\Omega)} = \rVert \rho_0 \rVert_{L^p(\Omega)}, \text{ for all } t\ge0,\ p\in [1,\infty].
 \end{equation}

\subsection{Evolution of the potential and kinetic energy}\label{sec21}
 Let us define the \emph{potential energy} and \emph{kinetic energy} of the solution as
 \begin{equation}\label{def_E} E_P(t) := \int_{\Omega} \rho(x,t)x_2 \,dx, \quad E_K(t):= \frac{1}{2}\int_{\Omega}|u(x,t)|^2\,dx.
 \end{equation}
 As we will see, the evolution of these energies plays a crucial role in the proof of Theorem~\ref{infinite_growth}.
 The rate of change of $E_P$ can be easily computed as
 \begin{equation}\label{derivative_potential}
 E_P'(t) = \int_\Omega \rho_t x_2\,dx = \int_\Omega -u\cdot (\nabla \rho) x_2\, dx = \int_{\Omega}\rho u_2\, dx,
 \end{equation}
 where the last equality follows from the divergence theorem and $\nabla \cdot u=0$, and note that the boundary integral in the divergence theorem is zero: in $\R^2$ it follows from $\rho(\cdot,t)$ having compact support, and in $\T^2$ it follows from the symmetries in \textbf{(A2)}.

 Similarly, one can compute the rate of change of the kinetic energy $E_K$ as \[
 E_K'(t) = -\int_\Omega {\rho u_2} dx - \nu \int_\Omega |\nabla  u |^2 dx.
 \]
 Combining the two equations, the total energy $E_P(t)+E_K(t)$ is non-increasing in time, and more precisely we have
 \begin{equation}\label{energy_conservation}
 E_P(t) + E_K(t) +  \nu  \int_0^{t}\int_\Omega | \nabla u(x,s) |^2 dxds = E_K(0) + E_P(0)\quad \text{ for all }t\ge0.
 \end{equation}

 From our discussion above, $\rho(\cdot,t)$  remains odd in $x_2$ for all $t\geq 0$, and the property  \textbf{(A3)} holds for all $t\geq 0$. Thus $E_P(t)$ is positive for all times. Combining this with \eqref{energy_conservation} gives
  \begin{equation}\label{energy_bound}
 0\leq E_P(t)\leq E_P(0)+E_K(0) \quad\text{ and }0\leq E_K(t)\leq E_P(0)+E_K(0) \quad\text{ for all }t\geq 0.
 \end{equation}
 In addition, using that $E_P(t)\geq 0, E_K(t)\geq 0$ for all $t\geq 0$, we can send $t\to\infty$ in \eqref{energy_conservation} to obtain
 \begin{equation}\label{dissipation}
  \nu  \int_0^\infty \|\nabla u(t)\|_{L^2(\Omega)}^2 dt \leq E_P(0)+E_K(0).
 \end{equation}
 In the next lemma we compute the second derivative of $E_P$, which will be used later.
\begin{lemma}\label{second_derivative}
Let $(\rho, u)$ be a solution to \eqref{Boussinesq} with initial data $(\rho_0,u_0)$ satisfying \textbf{\textup{(A1)}}--\textbf{\textup{(A3)}}. Then the potential energy $E_P$ defined in \eqref{def_E} satisfies
\begin{equation}\label{2nd_der}
 E_P''(t) = A(t) +   B(t) - \delta(t) \quad\text{ for all }t\geq 0,
\end{equation}
where
\begin{equation}\label{def_A}
A(t):=\sum_{i,j=1}^2\int_{\Omega}((-\Delta)^{-1}\partial_2\rho) \partial_i u_j\partial_{j}u_i  \,dx,~~ B(t):=  \nu  \int_\Omega \rho  \Delta u_2 dx, ~~and ~~ \delta(t):=\rVert \partial_{1}\rho \rVert_{\dot{H}^{-1}(\Omega)}^2.
\end{equation}
\end{lemma}
\begin{proof}
Differentiating \eqref{derivative_potential} in time, we get
\begin{equation}\label{2nd_temp}
\begin{split}
E_P''(t) &=\int_{\Omega}{-u\cdot \nabla(\rho u_2) + \rho\left( -\partial_2p - \rho +  \nu \Delta u_2 \right)}dx\\
&= \int_{\Omega}{\rho\left( -\partial_2p - \rho +  \nu \Delta u_2 \right)}dx,
\end{split}
\end{equation}
where the second inequality follows from the incompressibility of $u$ and the fact that the boundary integral is zero as we apply the divergence theorem: for $\Omega=\R^2$ it follows from $\rho(\cdot,t)$ having compact support, whereas for $\Omega = \T^2$ we are using $u\cdot n=0$ on the boundary of $ [-\pi,\pi]^2$ due to our symmetry assumptions in \textbf{(A2)}. 
Comparing \eqref{2nd_temp} with our goal \eqref{2nd_der}, it suffices to show that
\begin{equation}\label{claim1}
\int_{\Omega}\rho(-\partial_2 p - \rho) dx= A(t) -\delta(t).
\end{equation}
To do so, we take divergence in the equation for $u$ in \eqref{Boussinesq}. Using the incompressibility of $u$, we get $\nabla\cdot \left(u\cdot \nabla u\right) = -\Delta p - \partial_2\rho$, hence
\[
p = (-\Delta)^{-1} \nabla \cdot \left( u \cdot \nabla u \right) + (-\Delta)^{-1}\partial_2 \rho ,
\]
where $(-\Delta)^{-1}$ is the inverse Laplacian in $\Omega$ (which is either $\R^2$ or $\T^2$) defined in the standard way using Fourier transform (for $\Omega=\R^2$) or Fourier series (for $\Omega=\T^2$).
Therefore it follows that
\begin{align*}
-\partial_2p - \rho &= -\partial_2(-\Delta)^{-1} \nabla \cdot (u\cdot \nabla u) - (-\Delta)^{-1}\partial_{22}\rho  - \rho \\
& = -\sum_{i,j=1}^{2}\partial_2(-\Delta)^{-1}\left(\partial_{i}u_j\partial_{j}u_i\right) + (-\Delta)^{-1}\partial_{11}\rho.
 \end{align*}
This immediately yields that
\begin{align*}
 \int_\Omega \rho(-\partial_2p - \rho) dx & = -\sum_{i,j=1}^2 \int_\Omega \rho \partial_2(-\Delta)^{-1}\left(\partial_{i}u_j\partial_{j}u_i \right) dx + \int_\Omega \rho (-\Delta)^{-1}\partial_{11}\rho \,dx\\
 & = A(t) - \delta(t),
\end{align*}
where the second equality follows from integration by parts. This finishes the proof.
\end{proof}

The relation between $\delta(t)$ and $\|\rho(t)\|_{\dot{H}^s(\Omega)}$ has been investigated in \cite{KY}. Below we state the results from \cite{KY} and give a slightly improved estimate for the $\Omega=\mathbb{R}^2$ case.\footnote{In  \cite{KY}, the estimate corresponding to \eqref{eq_delta1} is \cite[equation (3.4)]{KY}, where an extra condition $\|\partial_1 \mu\|_{\dot{H}^{-1}}^2< \frac{1}{4}\|\mu\|_{L^2}^2$  was imposed.  In this lemma we give a slightly improved estimate where this assumption is dropped.} For the sake of completeness, we give a proof  in the appendix.  In the statement of the lemma we replace $\rho(t)$ by $\mu$, to emphasize that the estimate does not depend on the equation that $\rho(t)$ satisfies.

\begin{lemma}\label{smallinx1}
\textup{(a)} Assume $\Omega = \R^2$. Consider all $\mu \in C_c^{\infty}(\R^2)$ that is odd in $x_2$ and not identically zero. For all such $\mu$, there exists $c_1(s,\| \mu \|_{L^1},\| \mu \|_{L^2})>0$ such that 
\begin{equation}
\label{eq_delta1}
\| \mu\|_{\dot{H}^{s}(\R^2)} \ge c_1 \left(\| \partial_1 \mu \|_{\dot{H}^{-1}(\R^2)}^2\right)^{-\frac{s}{4}} \quad\text{ for all }s > 0.
\end{equation}

\noindent\textup{(b)} Assume $\Omega = \T^2$. Consider all $\mu \in C^{\infty}(\T^2)$ that is not identically zero, odd in $x_2$, even in $x_1$, with $\mu = 0$ on the $x_2$-axis, and $\mu\geq 0$ in $\T\times[0,\pi]$. For all such $\mu$, there exists $c_2(s, \int_{\T\times[0,\pi]} \mu^{1/3} dx)>0$ such that
\begin{equation}
\label{eq_delta2}
\| \mu \|_{\dot{H}^{s}(\T^2)} \ge c_2 \left(\| \partial_1 \mu \|_{\dot{H}^{-1}(\T^2)}^2\right)^{-s+\frac{1}{2}}\quad\text{ for all }s>\frac{1}{2}.
\end{equation}
\end{lemma}

  \subsection{Infinite-in-time growth of Sobolev norms}

Using Lemma~\ref{second_derivative}, for any $t_2>t_1\geq 0$, integrating $E_P''$ from $t_1$ to $t_2$, we get
\begin{equation}\label{energy_identity}
E_P'(t_2)-E_P'(t_1) = \int_{t_1}^{t_2} A(t)dt + \int_{t_1}^{t_2}B(t)dt - \int_{t_1}^{t_2}\delta(t)dt.
\end{equation}
In the next lemma we estimate the two integrals $\int_{t_1}^{t_2} A(t)dt$ and $\int_{t_1}^{t_2}B(t)dt$ on the right hand side.

\begin{lemma}\label{Abound} Assume $\nu>0$.
Let $(\rho, u)$ be a solution to \eqref{Boussinesq} with initial data $(\rho_0,u_0)$ satisfying \textbf{\textup{(A1)}}--\textbf{\textup{(A3)}}. Then for all $t_2>t_1\geq 0$, $A(t)$ defined in \eqref{def_A} satisfies
\begin{align}\label{atestimates}
 \int_{t_1}^{t_2} |A(t)| dt \le C(\rho_0) \int_{t_1}^{t_2} \rVert \nabla u(t) \rVert_{L^{2}(\Omega)}^2dt.
\end{align}
Furthermore, for all $s\geq 1$ and $t_2>t_1\geq 0$, $B(t)$ defined in \eqref{def_A} satisfies
\begin{align}\label{btestimates}
\int_{t_1}^{t_2} |B(t)|dt \le C(s,\rho_0) \nu  \left(\int_{t_1}^{t_2} \rVert \nabla u(t) \rVert_{L^{2}(\Omega)}^2dt\right)^{\frac12}  \left(\int_{t_1}^{t_2} \rVert \rho(t) \rVert^{\frac{2}s}_{\dot{H}^s(\Omega)}dt\right)^{\frac12}.
\end{align}
\end{lemma}
\begin{proof}
Let us show \eqref{atestimates} first.
Let $f := (-\Delta)^{-1}\partial_2\rho$, and we claim that 
\begin{align}\label{linftytoc}
\rVert f(\cdot,t)\rVert_{L^{\infty}(\Omega)} \le C(\rho_0) \quad\text{ for all }t\geq 0.
\end{align} Once this is proved, it follows that
\[
\int_{t_1}^{t_2} |A(t)|dt \le  \int_{t_1}^{t_2} \rVert f \rVert_{L^{\infty}(\Omega)}\rVert \nabla u \rVert_{L^2(\Omega)}^2dt \le C(\rho_0) \int_{t_1}^{t_2} \rVert \nabla u \rVert_{L^{2}(\Omega)}^2dt.
\]
To estimate $\rVert f\rVert_{L^{\infty}(\Omega)}$, we recall the following Hardy-Littlewood-Sobolev inequality for $\Omega=\mathbb{R}^2$ or $\mathbb{T}^2$: (when $\Omega=\mathbb{T}^2$, the function $g$ needs to satisfy an additional assumption $\int_{\Omega} g(x) dx=0$) 
\[
\rVert(-\Delta)^{-\frac{\alpha}{2}} g \rVert_{L^{q}(\Omega)} \le C(\alpha,p,q) \rVert g \rVert_{L^{p}(\Omega)}\quad \text{for $0<\alpha < 2$, $1 < p< q <\infty$, and $\frac{1}{q}=\frac{1}{p}-\frac{\alpha}{2}$}.
\]
We choose $\alpha=1$, $q=4$, $p=\frac{4}{3}$ and $g= (-\Delta)^{\frac{1}{2}}f(\cdot,t) $ (note that $g =  (-\Delta)^{-\frac{1}{2}}\partial_2 \rho$ indeed has mean zero when $\Omega=\mathbb{T}^2$), then the above inequality becomes
\begin{align*}
&\rVert f(\cdot,t) \rVert_{L^4(\Omega)} \le C \|(-\Delta)^{\frac12} f\|_{L^{\frac{4}{3}}(\Omega)}=C\| (-\Delta)^{-\frac{1}{2}}\partial_2\rho \|_{L^{\frac{4}{3}}(\Omega)} \le C\rVert \rho \rVert_{L^{\frac{4}{3}}(\Omega)} \le C(\rho_0),\end{align*}
and we also have
\[
\rVert (-\Delta)^{1/2}f(\cdot,t)\rVert_{L^{4}(\Omega)} = \| (-\Delta)^{-\frac{1}{2}}\partial_2\rho\|_{L^4(\Omega)} \le C \rVert \rho \rVert_{L^4(\Omega)} \le C(\rho_0).
\]
In the above two estimates, the second-to-last inequality in both equations is due to the Riesz transform being bounded in $L^{p}(\Omega)$ for $1<p<\infty$, and the last inequality in both equations comes from \eqref{lp_conservation}.
Combining these estimates together, we have 
\[
\|f(\cdot,t)\|_{W^{1,4}(\Omega)} \le C(\rho_0) \quad\text{ for all }t\geq 0.
\] Then the boundedness of $f$ follows immediately from Morrey's inequality $W^{1,4}(\Omega) \subset C^{0,\frac{1}{2}}(\Omega)$ for both $\Omega=\mathbb{R}^2$ and $\mathbb{T}^2$. This leads to $\|f(\cdot,t)\|_{L^\infty(\Omega)} \leq C \|f(\cdot,t)\|_{W^{1,4}(\Omega)} \le C(\rho_0)$ for all $t\geq 0$,
which proves \eqref{linftytoc}.

Now we turn to the estimate for $B(t)$. Applying the divergence theorem to the definition of $B(t)$ from \eqref{def_A}, we see that
\begin{equation}\label{bt_bound1}
\begin{split}
 \int_{t_1}^{t_2} |B(t)|dt &=   \nu   \int_{t_1}^{t_2} \left|  \int_{\Omega}\nabla \rho \cdot \nabla u_2 dx\right| dt
\\
&\le  \nu  \left(\int_{t_1}^{t_2}\rVert u (t)\rVert^2_{\dot{H}^1(\Omega)}dt \right)^{\frac12}\left( \int_{t_1}^{t_2} \rVert \rho(t)\rVert^2_{\dot{H}^1(\Omega)}dt\right)^{\frac12},
\end{split}
\end{equation}
where we used the Cauchy-Schwarz inequality in the last step. Using  the Gagliardo-Nirenberg interpolation inequality, we obtain
\[
\begin{split}
\int_{t_1}^{t_2}|B(t)|dt 
&\le  \nu  \left(\int_{t_1}^{t_2} \rVert \nabla u(t) \rVert_{L^{2}(\Omega)}^2dt\right)^{\frac12}  \left(\int_{t_1}^{t_2}  C(s) \rVert \rho(t)\rVert_{L^2}^{2(1-\frac{1}s)}\rVert \rho(t)\rVert_{\dot{H}^s(\Omega)}^{\frac{2}s}dt\right)^{\frac12}\\
&\le  C(s,\rho_0) \nu  \left(\int_{t_1}^{t_2} \rVert \nabla u (t)\rVert_{L^{2}(\Omega)}^2dt\right)^{\frac12}  \left(\int_{t_1}^{t_2} \rVert \rho(t) \rVert^{\frac{2}s}_{\dot{H}^s(\Omega)}dt\right)^{\frac12},
\end{split}
\]
where the last inequality follows from \eqref{lp_conservation}. This finishes the proof of \eqref{btestimates}.
\end{proof}

Now we are ready to prove Theorem~\ref{infinite_growth}.

\begin{proof}[\textbf{\textup{Proof of Theorem~\ref{infinite_growth}}}]
The main idea of the proof is to estimate all terms in \eqref{energy_identity} for $t_1=T$ and $t_2=2T$ for $T\gg 1$, and obtain a contradiction if $\sup_{t\in[T,2T]}\|\rho(t)\|_{\dot{H}^s}$ grows slower than certain power of $T$. 

First, to bound the left hand side of \eqref{energy_identity}, note that \eqref{derivative_potential} and
the Cauchy-Schwarz inequality yields
\begin{equation}
|E_P'(t)| \le \rVert \rho(t) \rVert_{L^2}\rVert u(t)\rVert_{L^2} \le \|\rho_0\|_{L^2} \sqrt{2E_K(t)} \leq C(\rho_0,u_0) < \infty \quad  \text{ for all }t\ge0,
\end{equation}
where the second inequality follows from \eqref{lp_conservation} and the definition of $E_K$ in \eqref{def_E}, and the third inequality follows from \eqref{energy_bound}. Thus
\begin{equation}\label{E_P_prime}
 |E_P'(2T)-E_P'(T)|\leq C_0(\rho_0,u_0)< \infty \quad\text{ for all } T>0.
 \end{equation}

Plugging the estimates \eqref{E_P_prime} and \eqref{atestimates} into the identity  \eqref{energy_identity}, we have
\begin{equation}\label{delta_upper1}
\int_T^{2T} \delta(t)dt \le C_0(\rho_0,u_0) + C_1(\rho_0) \int_{T}^{2T} \rVert \nabla u(t) \rVert_{L^{2}(\Omega)}^2dt + \int_T^{2T}|B(t)|dt \quad \text{ for all }T> 0.
\end{equation}
Next we will bound the two integrals on the right hand side from above, and $\int_T^{2T} \delta(t)dt$ from below.  Let us define
\[
\eta(T):= \int_{T}^{2T} \rVert \nabla u(t) \rVert_{L^{2}(\Omega)}^2dt \quad\text{ and }\quad M_s(T) := \sup_{t\in[T,2T]} \|\rho(t)\|_{\dot{H}^s(\Omega)}.
 \]
  Combining \eqref{energy_conservation} and \eqref{energy_bound} yields 
\[
 \int_0^{\infty} \|\nabla u(t)\|^2_{L^2(\Omega)} dt \leq \nu^{-1} C(\rho_0,u_0) < \infty,
\]
where we also used the assumption $\nu>0$.
This implies
\begin{equation}\label{liminf_eta}
\lim_{T\to\infty} \eta(T)=0.
\end{equation}
To bound $\int_T^{2T}|B(t)|dt$, using \eqref{btestimates} and the definitions of $\eta(T)$ and $M_s(T)$,
\begin{equation}\label{B_bd}
\begin{split}
\int_T^{2T}|B(t)|dt &\le C(s,\rho_0) \nu  \left(\int_{T}^{2T} \rVert \nabla u \rVert_{L^{2}(\Omega)}^2dt\right)^{\frac12}  \left(\int_{T}^{2T} \rVert \rho \rVert^{\frac{2}s}_{\dot{H}^s(\Omega)}dt\right)^{\frac12}\\
&\leq C_2(s,\rho_0, \nu)  \eta(T)^{\frac12} M_s(T)^{\frac{1}{s}} T^{\frac12} \quad\text{ for all }s\geq 1,\,  T>0.
\end{split}
\end{equation}

Next we will bound the integral $\int_T^{2T}\delta(t) dt$ from below. If $\Omega=\mathbb{R}^2$, the assumption \ref{assumption2_rho} allows us to apply Lemma~\ref{smallinx1}(a) to $\rho(\cdot,t)$ (and note that its $L^1$ and $L^2$ norms are preserved in time), so there exists $c_3(s,\rho_0)>0$ such that 
\begin{equation}\label{delta_ineq1}
\|\rho(t)\|_{\dot{H}^s(\mathbb{R}^2)} \geq c_3(s,\rho_0) \delta(t)^{-\frac{s}{4}} \quad \text{ for all } s>0,\,  t>0.
\end{equation}
And if $\Omega=\mathbb{T}^2$, using the assumptions \ref{assumption2_rho} and  \ref{assumption3_rho} (note that these assumptions imply that $\int_{\mathbb{T}\times[0,\pi]} \rho(x,t)^{1/3}dx$ is preserved in time), by Lemma~\ref{smallinx1}(b), there exists $c_4(s,\rho_0)>0$ such that
\begin{equation}\label{delta_ineq2}
\|\rho(t)\|_{\dot{H}^s(\mathbb{T}^2)} \geq c_4(s,\rho_0) \delta(t)^{-(s-\frac12)} \quad \text{ for all } s>\frac12,\, t>0.
\end{equation}

Let us rewrite the equations \eqref{delta_ineq1} and \eqref{delta_ineq2} above in a unified manner for the two cases $\Omega=\mathbb{R}^2$ and $\mathbb{T}^2$, so we do not need to repeat similar proofs twice. For $\Omega$ either being $\mathbb{R}^2$ or $\mathbb{T}^2$, let us define
\begin{equation}\label{def_alpha}
\alpha_\Omega := \begin{cases} \frac{s}{4} & \Omega=\mathbb{R}^2\\
s-\frac12 & \Omega = \mathbb{T}^2
\end{cases}, \quad 
\underline{s}_\Omega := \begin{cases} 0 & \Omega=\mathbb{R}^2\\
\frac12 & \Omega = \mathbb{T}^2
\end{cases}, \quad
c_\Omega(s,\rho_0) := \begin{cases} c_3(s,\rho_0) & \Omega=\mathbb{R}^2\\
c_4(s,\rho_0) & \Omega = \mathbb{T}^2
\end{cases}.
\end{equation}
With these notations, the equations \eqref{delta_ineq1} and \eqref{delta_ineq2} become
\begin{equation}\label{temp001}
\|\rho(t)\|_{\dot{H}^s(\Omega)} \geq c_\Omega(s,\rho_0) \delta(t)^{-\alpha_\Omega} \quad \text{ for all }  s> \underline{s}_\Omega,\, t>0.
\end{equation}
Combining  \eqref{temp001} with the definition of $M_s$ gives
\begin{equation}\label{delta_lower}
\begin{split}
\int_T^{2T}\delta(t)dt \ge  \int_T^{2T} c_\Omega^{\frac{1}{\alpha_\Omega}}\|\rho(t)\|_{\dot{H}^s(\Omega)}^{-\frac{1}{\alpha_\Omega}} dt \geq c_\Omega^{\frac{1}{\alpha_\Omega}} M_s(T)^{-\frac{1}{\alpha_\Omega}} T \quad \text{ for all }s>\underline{s}_\Omega,\, T>0.
\end{split}
\end{equation}

Applying the bounds \eqref{B_bd}, \eqref{delta_lower} and the definition of $\eta(T)$ to the inequality \eqref{delta_upper1} (and note that $\underline{s}_\Omega<1$), we have
\[
c_5 M_s(T)^{-\frac{1}{\alpha_\Omega}} T \leq C_0 + C_1 \eta(T) + C_2  \eta(T)^{\frac12} M_s(T)^{\frac{1}{s}} T^{\frac12} \quad\text{ for all }s\geq 1,\, T>0,
\]
where $c_5:=c_\Omega(s,\rho_0)^{\frac{1}{\alpha_\Omega}}, C_0 := C_0(\rho_0,u_0), C_1 := C_1(\rho_0)$ and $C_2 := C_2(s,\rho_0,\nu)$ -- note that they are all strictly positive and do not depend on $T$.
Rearranging the terms, the inequality is equivalent to
\begin{equation}\label{temp_000}
\left(c_5  - C_2\eta(T)^{\frac12} T^{-\frac12} M_s(T)^{\frac{1}{s}+\frac{1}{\alpha_\Omega}}\right)M_s(T)^{-\frac{1}{\alpha_\Omega}}T \leq C_0+C_1 \eta(T)  \quad\text{ for all }s\geq 1 ,\, T>0.
\end{equation}
We claim that this implies 
\begin{equation}\label{claim_M}
\limsup_{T\to\infty} T^{-\frac12} M_s(T)^{\frac{1}{s}+\frac{1}{\alpha_\Omega}}=+\infty \quad\text{ for all }s\geq 1.
\end{equation}
Towards a contradiction, assume 
\[
A := \limsup_{T\to\infty} T^{-\frac12} M_s(T)^{\frac{1}{s}+\frac{1}{\alpha_\Omega}} < \infty \quad\text{ for some }s\geq 1.
\] Combining this assumption with \eqref{liminf_eta} gives 
\[
\begin{split}
\limsup_{T\to\infty} \eta(T)^{\frac12} T^{-\frac12} M_s(T)^{\frac{1}{s}+\frac{1}{\alpha_\Omega}} &=  \left(\limsup_{T\to\infty} T^{-\frac12} M_s(T)^{\frac{1}{s}+\frac{1}{\alpha_\Omega}}\right) \left(\lim_{T\to\infty} \eta(T)^{\frac12}\right)=0,
\end{split}
\]
so the parenthesis in \eqref{temp_000} converges to $c_5$ as $T\to\infty$. For the remaining term on the left hand of \eqref{temp_000}, we have
\begin{equation}
\begin{split}
\liminf_{T\to\infty} M_s(T)^{-\frac{1}{\alpha_\Omega}}T &= \liminf_{T\to\infty} \left(T^{-\frac12} M_s(T)^{\frac{1}{s}+\frac{1}{\alpha_\Omega}} \right)^{-\frac{s}{s+\alpha_\Omega}} T^{\frac{s+2\alpha_\Omega}{2(s+\alpha_\Omega)}}\\
& = \liminf_{T\to\infty} A^{-\frac{s}{s+\alpha_\Omega}} T^{\frac{s+2\alpha_\Omega}{2(s+\alpha_\Omega)}} = +\infty.
\end{split}
 \end{equation}
The above discussion yields that the liminf  of the left hand side of \eqref{temp_000} is $+\infty$. This contradicts \eqref{liminf_eta}, which says the right hand side of \eqref{temp_000} goes to $C_0<\infty$ as $T\to\infty$. This finishes the proof of the claim \eqref{claim_M}.

Finally, using the definition of $M_s$, we have that \eqref{claim_M} is equivalent to
\[
\limsup_{t \to\infty} t^{-\frac12} \|\rho(t)\|_{\dot{H}^s}^{\frac{1}{s}+\frac{1}{\alpha_\Omega}} = +\infty.
\]
Recalling the definition of $\alpha_\Omega$ from \eqref{def_alpha}, we see that the desired estimates \eqref{growth1} and \eqref{growth2} follow immediately.
\end{proof}

Although it is unclear whether the algebraic rates are sharp, in the next proposition we show that under the assumptions \textbf{\textup{(A1)}}--\textbf{\textup{(A3)}},  $\|\rho(t)\|_{H^1(\Omega)} $ can at most have sub-exponential growth.

\begin{proposition} \label{prop_subexp}
Let $ \Omega=\mathbb{R}^2$ or $\mathbb{T}^2$. For any initial data $(\rho_0,u_0)$ satisfying \textbf{\textup{(A1)}}--\textbf{\textup{(A3)}}, $\|\rho(t)\|_{H^1(\Omega)}$ satisfies the sub-exponential bound
\[
\|\rho(t)\|_{H^1(\Omega)} \lesssim \exp (Ct^{\alpha})\quad\text{ for all }t>0
\] 
for some constant $\alpha\in(0,1)$.
\end{proposition}

\begin{proof}
For both $\Omega=\mathbb{T}^2$ and $\mathbb{R}^2$, 
$\|\nabla\rho(t)\|_{L^2(\Omega)}$ satisfies the estimate (see \cite[Eq. (2.37)]{KW2020})
\begin{equation}\label{rho_up_bd}
\|\nabla \rho(t)\|_{L^2(\Omega)} \lesssim \exp\left(\int_0^t \|\nabla u(s)\|_{L^\infty(\Omega)} ds\right) \|\nabla\rho_0\|_{L^2(\Omega)}.
\end{equation}
  Recall that \eqref{dissipation} gives
 \begin{equation}\label{l2bd0}
 \int_0^\infty \|\nabla u(t)\|_{L^2}^2 dt \leq C(\nu,\rho_0,u_0).
 \end{equation}
Combining this with the Gagliardo--Nirenberg inequality 
\begin{equation}\label{ineqGN}
\|\nabla u\|_{L^\infty(\Omega)} \leq \|\nabla u\|_{L^2(\Omega)}^{\frac{p-2}{2p-2} } \|\nabla^2 u\|_{L^p(\Omega)}^{\frac{p}{2p-2} } \quad\text{ for } p> 2
\end{equation}
 and H\"older's inequality, the exponent in \eqref{rho_up_bd} can be bounded above by
\begin{equation}\label{nabla_bd_temp}
\int_0^t \|\nabla u(s)\|_{L^\infty(\Omega)} ds\leq C(p,\nu,\rho_0,u_0) \left(\int_0^t  \|\nabla^2 u\|_{L^p(\Omega)}^{\frac{2p}{3p-2}} ds\right)^{\frac{3p-2}{4p-4}}.
\end{equation}
When $\Omega=\mathbb{T}^2$, by \cite[Theorem 2.1]{KW2020}, $\|u(t)\|_{W^{2,p}}<C(p,\nu, \rho_0,u_0)$ for all $p<\infty$.  So one can choose $p\gg 1$ to obtain the sub-exponential upper bound 
\begin{equation}\label{subexp}
\|\nabla \rho(t)\|_{L^2} \leq C(\rho_0) \exp \left(C(\epsilon, \nu,\rho_0,u_0)t^{\frac{3}{4}+\epsilon}\right)\quad\text{ for any }\epsilon>0, t>0.
\end{equation}
Next we move on to the $\Omega=\mathbb{R}^2$ case. 
Combining \eqref{l2bd0} (and recall $\|\omega\|_{L^2}=\|\nabla u\|_{L^2}$) with the estimate $\frac{d}{dt} \|\omega(t)\|_{L^2(\mathbb{R}^2)}^2 < C(\nu,\rho_0)$ (see the equation before (3.1) in \cite{KW2020}), we have
\begin{equation}\label{omega_l2_bd}
 \|\omega(t)\|_{L^2(\mathbb{R}^2)} < C(\nu,\rho_0,u_0) \quad\text{ for all }t\geq 0.
\end{equation}
Following the notations from  \cite{KW2020}, let us define $\zeta = \omega - \partial_1 (I-\Delta)^{-1}\rho$ to be the modified vorticity. Since one has $\|\partial_1 (I-\Delta)^{-1}\rho\|_{W^{1,p}} \leq C(p,\rho_0)$ for all $1<p<\infty$, it implies  
\begin{equation}\label{zeta}
\|\zeta-\omega\|_{L^p} \leq  C(p,\rho_0)\quad\text{  and }\quad\|\nabla\zeta-\nabla\omega\|_{L^p} \leq  C(p,\rho_0).
\end{equation} Combining \eqref{omega_l2_bd} and \eqref{zeta} gives a uniform-in-time bound $\|\zeta(t)\|_{L^2}< C(\nu,\rho_0,u_0)$. Defining $\psi_p(t) := \int_{\mathbb{R}^2} |\nabla\zeta(t)|^p$ for $p\geq 2$, \cite[Eq.(3.2)]{KW2020} gives
\[
\psi_2'(t) + \frac{\psi_2^2}{\|\zeta\|_{L^2}^2} \leq C\psi_2 + C, 
\]
thus the above uniform-in-time bound for $\|\zeta\|_{L^2}^2$ gives a uniform-in-time bound for $\psi_2(t)$. For any $2\leq p < \infty$, \cite[Eq.(3.3)]{KW2020} gives
\[
\psi_{2p}'(t) + \frac{\psi_{2p}^2}{C\psi_p^2} \leq Cp^2 \psi_{2p} + Cp^5 \psi_{2p}^{(p-1)/p}.
\]
One can use induction (for $p=2, 4, 8,\dots$) to obtain a uniform-in-time bound $\psi_p(t) \leq C(p,\nu,\rho_0,u_0)$, and combining this bound with \eqref{zeta} gives 
\[
\|\nabla^2 u(t)\|_{L^p} \leq C(p) \|\nabla \omega(t)\|_{L^p} \leq C(p) (\|\nabla\zeta(t)\|_{L^p} + C(p,\rho_0)) \leq C(p,\nu,\rho_0,u_0).
\]
Finally, choosing an arbitrarily large $p\gg 1$ and plugging the above uniform-in-time estimate into \eqref{nabla_bd_temp}, we again have the sub-exponential upper bound \eqref{subexp} for $\Omega=\mathbb{R}^2$.
\end{proof}

\section{Infinite-in-time growth for inviscid Boussinesq and 3D Euler}

\subsection{Vorticity lemma for flows with fixed kinetic energy}
Before proving the main theorems, let us start with a simple observation. It says that for any vector field $u$ in a square $Q = [0,\pi]^2$ with a fixed kinetic energy, if its vorticity integral $A:= \int_Q \omega dx$ is big, then for $1<p\leq \infty$, $\|\omega\|_{L^p}$ must be even bigger, at least of order $A^{3-\frac{2}{p}}$.

\begin{lemma}\label{omega_lp}
Let $Q := [0,\pi]^2$. For any vector field $u \in C^\infty(Q)$, let $\omega := \partial_1 u_2 - \partial_2 u_1$. Let us denote $E_0:=\int_Q |u|^2 dx$, and 
$
A:=\int_Q \omega(x) dx 
$. Then we have the following lower bound for  $\|\omega\|_{L^p(Q)}$:
\begin{equation}\label{omega_lower}
\|\omega\|_{L^p(Q)}\geq c_0 \max\Big\{ E_0^{-1+\frac{1}{p}}|A|^{3-\frac{2}{p}} , |A|\Big\} \quad\text{ for all }p\in[1,\infty],
\end{equation}
where $c_0 = (128\pi^2)^{-1} > 0 $ is a universal constant.
\end{lemma}

\begin{proof} Without loss of generality, assume $A>0$. (If $A<0$, we can prove the estimate for $-u$, whose vorticity integral would be positive).
By Green's theorem, we have
\[
\int_{\partial Q} |u(x)|ds \geq \int_{\partial Q} u(x) \cdot dl = \int_{ Q} \omega(x) dx = A,
\]
where the integral in $ds$ denotes the (scalar) line integral with respect to arclength, and the integral in $dl$ denotes the (vector) line integral counterclockwise along $\partial Q$.

For any $ r \in[0,\frac{\pi}{2})$, let us define \[
Q_ r  := [ r , \pi - r ] \times [ r , \pi - r ].
\] Note that $Q_0=Q$, and $Q_ r $ shrinks to a point as $ r \nearrow\frac{\pi}{2}$. Let us denote
\[
 r_0 := \inf\left\{ r \in \big[0,\frac{\pi}{2}\big): \int_{\partial Q_ r} |u(x)|ds = \frac{A}{2}\right\}.
\]
Since $\int_{\partial Q_0} |u(x)|ds > A$, and $\int_{\partial Q_ r} |u(x)|ds \to 0$ as $ r \nearrow \frac{\pi}{2}$, the above definition leads to a well-defined $ r_0 \in (0,\frac{\pi}{2})$, and in addition we have 
\[
\int_{\partial Q_ r} |u(x)|ds > \frac{A}{2} \quad \text{ for all } r\in[0, r_0).
\]
Next we claim that 
\begin{equation}
\label{delta0}
 r_0 < 16\pi E_0 A^{-2}.
\end{equation} 
To show this, note that for all $0< r< r_0$, we can apply the Cauchy-Schwarz inequality on $\partial Q_{ r}$ (and use $|\partial Q_{ r}|<4\pi$) to obtain
\[ \int_{\partial Q_ r} |u|^2 \, ds \geq \frac{1}{4\pi} \left(\int_{\partial Q_ r} |u|\,ds \right)^2 >  \frac{A^2}{16\pi}. \]   
Integrating the above inequality for $ r\in(0, r_0)$ over the direction transversal to $\partial Q_ r$ (and note that $\cup_{ r\in(0, r_0)} \partial Q_{ r} = Q\setminus Q_{ r_0}$), we obtain that 
\[ E_0 \geq \int_{Q \setminus Q_{ r_0}} |u|^2 \,dx = \int_0^{ r_0}\int_{\partial Q_ r} |u|^2 \, ds d r > \frac{A^2 r_0}{16\pi},
\]
which yields the claim \eqref{delta0}. 
Note that \eqref{delta0} implies 
\begin{equation}\label{area}
|Q\setminus Q_{ r_0}| = \int_0^{ r_0}|\partial Q_ r| dr \leq \min\{ 4\pi  r_0, \pi^2\} \leq \min\{ 64\pi^2 E_0 A^{-2}, \pi^2\}.
\end{equation}
By Green's theorem and the definition of $ r_0$, 
\begin{equation}\label{intw}
\int_{Q\setminus Q_{ r_0}} \omega \,dx = \int_{\partial Q} u \cdot dl - \int_{\partial Q_{ r_0}} u \cdot dl \geq A-\frac{A}{2} = \frac{A}{2}. 
\end{equation}
Finally,  we apply H\"older's inequality to  bound $\|\omega\|_{L^p(Q)}$ from below for $p\in[1,\infty]$:
\[
\|\omega\|_{L^p(Q)} \geq \|\omega\|_{L^p(Q\setminus Q_{ r_0})} \geq \left(\int_{Q\setminus Q_{ r_0}} \omega \,dx\right) |Q\setminus Q_{ r_0}|^{-1+\frac{1}{p}}\quad\text{ for all }p\in[1,\infty].
\]
Applying the estimates \eqref{intw} and  \eqref{area} into the above inequality finishes the proof of \eqref{omega_lower} with a universal constant $c_0 = (128\pi^2)^{-1}$.
\end{proof}

\subsection{Infinite-in-time growth for inviscid Boussinesq equations}
Now we are ready to prove the infinite-in-time growth results. Let us start with Theorem \ref{invT2} for $\Omega=\mathbb{T}^2$. 

\begin{proof}[\textbf{\textup{Proof of Theorem \ref{invT2}}}] Using the Biot-Savart law $u = \nabla^\perp(-\Delta)^{-1}\omega$, one can easily check that in $\mathbb{T}^2=(-\pi,\pi]^2$, the even-odd symmetry of $\rho$ and odd-odd symmetry of $\omega$ is preserved for all times. This implies the odd-even symmetry of $u_1$ and even-odd symmetry of $u_2$ hold for all times. In particular, denoting 
\[
Q :=  [0,\pi] \times [0,\pi],
\] we have $u\cdot n=0$ on $\partial Q$ for all times.

For any $x\in\mathbb{T}^2$ and $t\geq 0$, let $\Phi_t(x)$ be the flow map defined by 
\[
\partial_t \Phi_t(x) = u(\Phi_t(x), t),\quad \Phi_0(x)=x.
\]
Using $u\cdot n=0$ on $\partial Q$ for all times (and $u=0$ at the four corners of $\partial Q$), for any $x\in\partial Q$, $\Phi_t(x)$ remains on the same side of $\partial Q$ for all times during the existence of a smooth solution. 
Combining this with the fact that $\rho$ is preserved along the flow map, the assumptions on $\rho_0$ implies
\begin{equation}\label{rho_sign}
\rho(0,x_2, t)\geq 0  ~~~\text{ and }~~~  \rho(\pi,x_2, t)\leq 0 \quad \text{ for all }x_2\in[0,\pi], t\geq 0.
\end{equation}

Note that the  odd-in-$x_2$ symmetry of $\rho_0$ yields $\rho_0(0,0)=\rho_0(0,\pi)=0$, so the supremum in
$
k_0 := \sup_{x_2\in[0,\pi]}\rho_0(0,x_2)>0
$ is achieved at some $\rho(0,a)$ for $a\in(0,\pi)$. In addition, by continuity of $\rho_0$, there exists some $b\in(0,a)$ such that $\rho_0(0,b) = k_0/2$ and $\rho_0\geq k_0/2$ on $\{0\}\times[b,a]$.

Since $u\cdot n=0$ on $\partial Q$ for all times, $\Phi_t(0,a)$ and $\Phi_t(0,b)$ remain on the line segment $\{0\}\times(0,\pi)$ for all times.  Denote 
\begin{equation}\label{hdef1} 
h(t) := |\Phi_t(0,b)-\Phi_t(0,a)|,
\end{equation}
which is strictly positive as long as $u$ remains smooth. Note that $\rho(\Phi_t(0,a),t)=k_0$ and $\rho(\Phi_t(0,b),t)=k_0/2$ for all times.
This implies
\begin{equation}\label{rhoc}
\|\nabla \rho(t)\|_{L^\infty(Q)} \geq \frac{|\rho(\Phi_t(0,b),t) - \rho(\Phi_t(0,a),t)|}{|\Phi_t(0,b)-\Phi_t(0,a)|} \geq \frac{k_0}{2} h(t)^{-1}
\end{equation}
for all times during the existence of a smooth solution.

Next let us define 
\[A(t) := \int_Q \omega(x,t) dx,
\] and we make a simple but useful observation about the monotonicity of $A(t)$. Using the symmetries and the facts $\nabla \cdot u =0$ in $Q$ and $u\cdot n=0$ on $\partial Q$, we find that 
\begin{equation}\label{w_mono}
\begin{split}
A'(t) 
&= - \int_Q u(x,t) \cdot \nabla \omega (x,t)\,dx - \int_Q \partial_{x_1} \rho (x,t)\,dx\\
&= \int_0^\pi \rho(0,x_2,t) dx_2 - \int_0^\pi \rho(\pi,x_2,t) dx_2\\
& \geq \frac{k_0}{2} h(t),
\end{split}
\end{equation}
where the inequality follows from \eqref{rho_sign}, the definition of $h(t)$, and the fact that $\rho(\cdot,t)\geq k_0/2$ on the line segment connecting $\Phi_t(0,a)$ and $\Phi_t(0,b)$.
We now integrate \eqref{w_mono} in $[0,t]$ and apply \eqref{rhoc}. 
This leads to
\begin{equation}\label{ineqA}
A(t)  \geq \frac{k_0^2}{4}\int_0^t \|\nabla \rho(\tau)\|_{L^\infty(Q)}^{-1}d \tau + A(0) .
\end{equation}

In order to apply Lemma~\ref{omega_lp}, we need to bound $\|u(t)\|_{L^2(Q)}^2$ from above.
From the same calculation in Section \ref{sec21}, the sum of the kinetic and potential energies is conserved in $\mathbb{T}^2$, hence it is also conserved in $Q$ due to the symmetries: 
\[ \frac12 \int_{Q} |u(x,t)|^2 \,dx + \int_{Q} x_2 \rho(x,t) \,dx = \frac12 \int_{Q} |u_0(x)|^2 \,dx + \int_{Q} x_2 \rho_0(x) \,dx. \]     
Since $\rho$ is advected by the flow, $\|\rho(t)\|_{L^1(Q)}$ is conserved in time, so $|\int_{Q} x_2 \rho(x,t) \,dx|\leq \pi \|\rho_0\|_{L^1(Q)}$ for all times. This implies
\[
\int_{Q} |u(x,t)|^2 \,dx \leq  \int_{Q} |u_0(x)|^2 \,dx + 4\pi \|\rho_0\|_{L^1(Q)} =: E_0(\rho_0,u_0)
\] 
for all times. Now we can apply Lemma~\ref{omega_lp} with $p=+\infty$ to conclude
\begin{equation}\label{omlb1} \|\omega(t)\|_{L^\infty} \geq c_0 E_0^{-1}A(t)^3 \geq  c_0 E_0^{-1} \left(\frac{k_0^2}{4}\int_0^t \|\nabla \rho(\tau)\|_{L^\infty}^{-1}d \tau + A(0)\right)^3,
 \end{equation}
where we used \eqref{ineqA} in the last step. Note that $A(0)$ maybe positive or negative.

On the other hand, the Lagrangian form of the evolution equation for vorticity  
\[ \frac{d}{dt} \omega(\Phi_t(x),t) = -\partial_{x_1} \rho(\Phi_t(x),t) \]
implies that 
\begin{equation}\label{omup}
\|\omega(t)\|_{L^\infty}  \leq \int_0^t \|\nabla \rho(\tau)\|_{L^\infty} d\tau + \|\omega_0\|_{L^\infty}.
\end{equation} 
Combining \eqref{omlb1} and \eqref{omup}, we arrive at
\begin{equation}\label{rhofin} 
\int_0^t \|\nabla \rho(\tau)\|_{L^\infty} d\tau + \|\omega_0\|_{L^\infty} \geq c_0 E_0^{-1} \left(\frac{k_0^2}{4} \int_0^t \|\nabla \rho(\tau)\|_{L^\infty}^{-1}d \tau + A_0\right)^3. \end{equation}
Let us denote
\[
F(t) := \int_0^t \|\nabla \rho(\tau)\|_{L^\infty} d\tau.
\]
Since Cauchy--Schwarz inequality yields
\[
\int_0^t \|\nabla \rho(\tau)\|_{L^\infty}^{-1} d \tau \geq t^2 \left(\int_0^t \|\nabla \rho(\tau)\|_{L^\infty} d\tau\right)^{-1}\geq t^2 F(t)^{-1}\quad\text{ for all }t>0,
\]
plugging it into \eqref{rhofin} gives an inequality relating $F(t)$ with itself:
\begin{equation}\label{ineq_F}
F(t) \geq c_0 E_0^{-1} \left( \frac{k_0^2}{4} t^2 F(t)^{-1} + A_0\right)^3 - \|\omega_0\|_{L^\infty}.
\end{equation}
Our goal is to show that there exists some $c_1(\rho_0,\omega_0)>0$ such that
\begin{equation}\label{claim_F}
F(t)\geq c_1(\rho_0,\omega_0) t^{3/2} \quad\text{ for all }t\geq 1.
\end{equation}
Towards a contradiction, suppose \eqref{claim_F} does not hold at some $t_1\geq 1$, so $t_1^2 F(t_1)^{-1} \geq c_1^{-1} t_1^{1/2} $. Since $t_1\geq 1$, one can choose $c_1$ sufficiently small (only depending on initial data) such that the right hand side of \eqref{ineq_F} is bounded below by $4^{-4} c_0 E_0^{-1} k_0^6 c_1^{-3}  t_1^{3/2}$. 
On the other hand, the left hand side is bounded above by $c_1 t_1^{3/2}$. Thus we obtain a contradiction if we further require $c_1< 4^{-1} (c_0 E_0^{-1} k_0^6)^{1/4}$.

Finally, note that \eqref{claim_F} directly implies $\sup_{\tau\in[0,t]}\|\nabla\rho(\tau)\|_{L^\infty} \geq c_1(\rho_0,\omega_0) t^{1/2} $ for all $t\geq 1$. For $t\in(0, 1)$, recall that the definition of $k_0$ and the fact $\rho(0,0,t)=0$ yield $\|\nabla\rho(t)\|_{L^\infty}\geq k_0/\pi \geq (k_0/\pi) t^{1/2}$ for $t\in(0,1)$. Combining these two estimates finishes the proof. 
\end{proof}

\begin{remark}
 Theorem~\ref{invT2} does not give us any infinite-in-time growth result for $\omega(\cdot,t)$. All we have is the following conditional growth estimate coming from  \eqref{omlb1}: If $\limsup_{t\to\infty} t^{-1} \|\nabla\rho(t)\|_{L^\infty} < \infty$, this must imply $\lim_{t\to\infty}\|\omega(t)\|_{L^\infty}=\infty$.
\end{remark}

\medskip
\begin{proof}[\textbf{\textup{Proof of Theorem \ref{invB}}}]
The proof is similar to the previous one, and in fact it is easier due to the uniform positivity of $\rho_0$ on $\{0\}\times[0,\pi]$. Using the Biot-Savart law, one can check the even-in-$x_1$ symmetry of $\rho$ and odd-in-$x_1$ symmetry of $\omega$ is preserved for all times. Denoting $Q := [0,\pi]\times[0,\pi]$, the symmetries and the boundary condition yield that $u\cdot n=0$ on $\partial Q$ for all times. In particular, this implies 
\begin{equation}\label{two_bounds_rho}
\rho(0,x_2,t)\geq k_0>0\text{ and }\rho(\pi,x_2,t)\leq 0\quad\text{ for all }x_2\in[0,\pi], t\geq 0
\end{equation} during the existence of a smooth solution. 

Again, let us define $A(t) := \int_Q \omega(x,t) dx$. A calculation similar to the previous proof shows that in this case 
\[
A'(t) \geq \int_0^\pi \rho(0,x_2,t) dx_2 - \int_0^\pi \rho(\pi,x_2,t) dx_2 \geq k_0\pi,
\]
where the last inequality follows from \eqref{two_bounds_rho}. This gives us a lower bound
\begin{equation}\label{ineq_A}
A(t) \geq k_0\pi t + A(0) \quad\text{ for all }t\geq 0.
\end{equation}
An identical argument as in the proof of Theorem~\ref{invT2} gives $\int_\Omega |u(x,t)|^2\,dx\leq E_0(\rho_0,u_0)$ uniformly in time, thus we can apply Lemma~\ref{omega_lp} to obtain
\begin{equation}\label{omega_lp2}
\|\omega\|_{L^p(Q)} \geq c_0 E_0^{-1+\frac{1}{p}}|A(t) |^{3-\frac{2}{p}} \quad\text{ for all }p\in[1,\infty].
\end{equation}
Also, note that Green's theorem yields
\begin{equation}\label{bd_u}
A(t) = \int_{\partial Q} u\cdot dl \leq  4\pi\|u(t)\|_{L^\infty} .
\end{equation}
Regarding the growth of $ \nabla \rho$, note that \eqref{omup} still holds in a strip, so
\begin{equation}\label{rho_lower}
\sup_{\tau \in [0,t]}\|\nabla \rho(\tau)\|_{L^\infty} \geq t^{-1} (\|\omega(t)\|_{L^\infty}-\|\omega_0\|_{L^\infty}) \quad\text{ for all }t>0.
\end{equation}

Below we discuss two cases:

\noindent \textbf{Case 1.} $A(0)\geq 0$. In this case \eqref{ineq_A} gives 
\[
A(t) \geq k_0\pi t \quad\text{ for all }t>0.
\] 
We then apply \eqref{omega_lp2} and \eqref{bd_u} to obtain lower bounds for $\|\omega(t)\|_{L^p(Q)}$ and $\|u(t)\|_{L^\infty}$:
\begin{align}\label{w_bd_p}
\|\omega(t)\|_{L^p(Q)} &\geq c_1(\rho_0,\omega_0) t^{3-\frac{2}{p}} \quad\text{ for all }p\in[1,+\infty], t\geq 0,\\
\|u(t)\|_{L^\infty(Q)} &\geq \frac{1}{4}k_0 t \hspace{2cm}\text{ for all }t\geq 0.
\end{align}

Regarding the growth of $\nabla\rho$, we apply \eqref{w_bd_p} with $p=+\infty$ and combine it with \eqref{rho_lower} to obtain
\[
\sup_{\tau \in [0,t]}\|\nabla \rho(\tau)\|_{L^\infty} \geq  t^{-1} \big(c_1(\rho_0,\omega_0) t^3 - \|\omega_0\|_{L^\infty}\big),
\]
which implies 
\[
\sup_{\tau \in [0,t]}\|\nabla \rho(\tau)\|_{L^\infty} \geq c_1(\rho_0,\omega_0) t^2  \quad\text{for all }t\geq \left(\frac{\|\omega_0\|_{L^\infty}}{c_1(\rho_0,\omega_0)}\right)^{1/3}.
\]
Combining this large time estimate with the trivial lower bound $\|\nabla \rho(t)\|_{L^\infty}\geq \frac{k_0}{\pi}$ for all times, there exists some $c_2(\rho_0,\omega_0)>0$ such that
\begin{equation}\label{rho_temp}
\sup_{\tau \in [0,t]}\|\nabla \rho(\tau)\|_{L^\infty} \geq c_2(\rho_0,\omega_0) t^2  \quad\text{for all }t\geq 0.
\end{equation}

\noindent \textbf{Case 2.} $A_0<0$. In this case the right hand side of \eqref{ineq_A} becomes positive for $t>|A_0|/(k_0\pi)$. In addition, we have
\[
A(t) \geq \frac{1}{2}k_0\pi t \quad\text{ for all }t\geq T_0 =: \frac{2|A_0|}{k_0\pi}.
\]
Once we obtain this  (positive) linear lower bound for $t\geq T_0$, we can argue as in Case 1 to obtain lower bounds for $\|\omega(t)\|_{L^p(Q)}$, $\|u(t)\|_{L^\infty}$ and $\sup_{\tau \in [0,t]}\|\nabla \rho(\tau)\|_{L^\infty}$ for all $t\geq T_0$. In addition, combining the lower bound for $\|\nabla \rho(t)\|_{L^\infty}$ for $t\geq T_0$ with the trivial lower bound $\|\nabla \rho(t)\|_{L^\infty}\geq \frac{k_0}{\pi}$ for all times, we again have \eqref{rho_temp} with a smaller coefficient $c(\rho_0,\omega_0)>0$ that only depends on the initial data.
\end{proof}

\begin{remark}\label{rmk_stable}
If the assumptions on symmetries of $\rho_0$ and $\omega_0$ are dropped, the following simple argument still gives $\|\omega(t)\|_{L^1}\gtrsim t$ for $t\gg 1$. Let $Q_t := \{\Phi_t(x): x\in[0,\pi]\times[0,\pi]\}$, and denote by $\Gamma_t^1 := \{\Phi_t(x): x\in \{0\} \times[0,\pi]\}$ and $\Gamma_t^2 := \{\Phi_t(x): x\in\{\pi\}\times[0,\pi]\}$ the left and right boundary of $Q_t$. (Since $u\cdot n=0$ on $\partial \Omega$, the top and bottom boundaries of $Q_t$ remain on $\partial \Omega$ for all times). In addition, since $\rho$ is preserved along the flow,  at each $t$ we have $\rho(\cdot,t)|_{\Gamma_t^1} \geq k_0>0$, and $\rho(\cdot,t)|_{\Gamma_t^2} \leq 0$. Thus a computation similar to \eqref{w_mono} in the moving domain $Q_t$ gives
 \[
\frac{d}{dt}  \int_{Q_t} \omega(x,t) dx  = \int_{Q_t}-\partial_{x_1}\rho(t)dx \geq k_0 \pi \quad\text{ for all }t\geq 0,
\]
Therefore, as long as the solution $(\rho,\omega)$ remains smooth, we have
\begin{equation}\label{omega_l1}
\lVert \omega(t)\rVert_{L^1} \ge \int_{Q_t}\omega(x,t) dx \geq k_0\pi t - \|\omega_0\|_{L^1} \quad\text{ for all }t\geq 0.
\end{equation}
However, since $Q_t$ is in general largely deformed from a square for $t\gg 1$, we are not able to apply Lemma~\ref{omega_lp} to obtain faster growth rate for higher $L^p$ norms.

Note that given any  steady state $\omega_s$ of 2D Euler on the strip $\Omega$, $(0,\omega_s)$ is automatically a steady state of the inviscid Boussinesq equations \eqref{Boussinesqw}. Thus the infinite-in-time growth estimate \eqref{omega_l1} directly implies that any such steady state (with zero density) is nonlinearly unstable, in the sense that for any $0<k_0\ll 1$, an arbitrarily small perturbation $\rho_0=k_0 \cos(x_1), \omega_0=\omega_s$ leads to $\lim_{t\to\infty}\lVert \omega(t)\rVert_{L^1} =\infty$. See \cite{bedrossian2021nonlinear, MR3974167, MR4322283, MR3815212, MR4451473, MR4097325, zillinger2022stability} for more results on stability/instability of steady states of the inviscid or viscous Boussinesq equations.
\end{remark}

\subsection{Application to 3D axisymmetric Euler equation}

In this subsection we will prove Theorem~\ref{3dEth}, whose proof is a close analog of Theorem~\ref{invB}.

\begin{proof}[\textbf{\textup{Proof of Theorem~\ref{3dEth}}}] Using the Biot-Savart law, one can easily check that $\omega^\theta$ remains odd in $z$ and $u^\theta$ remains even in $z$ for all times while the solution stays smooth. Combining these symmetries with the Biot--Savart law \eqref{bs_euler} gives $u^z=0$ for $z=0$ and $z=\pi$ for all times.  
For a point $x$ on the $rz$-plane, let us define the flow-map $\Phi_t(x): [\pi,2\pi]\times\mathbb{T} \to [\pi,2\pi]\times\mathbb{T}$, given by
\[
\frac{d}{dt} \Phi_t(x) = (u_r(\Phi_t(x),t), u_z(\Phi_t(x),t)).
\]
Since $u_z=0$ on $z=\pi$, for any $x\in[\pi,2\pi]\times\{\pi\}$, we have $\Phi_t(x)$ remains on $[\pi,2\pi]\times\{\pi\}$.
From the first equation in \eqref{3dE}, we have $ru^{\theta}$ is conserved along the trajectory. Thus for any point $(r,\pi)$ with $r \in [\pi,2\pi]$, we have
\[
ru^{\theta}(r,\pi,t) \geq \pi u^\theta_0(\Phi^{-1}_t(r,\pi),0) \geq  \pi k_0,
\]
where the last inequality follows from the assumption $u^\theta_0\geq k_0>0$ on $z=\pi$ and the fact that $\Phi^{-1}_t(r,\pi) \in [\pi,2\pi]\times\{\pi\}$.
This implies
\begin{equation}\label{u_bd1}
u^{\theta}(r,\pi,t) \geq \frac{1}{2}k_0>0 \quad\text{ for all }r\in[\pi,2\pi], t\geq 0.
\end{equation}
Applying a similar argument for $z=0$, the assumption $|u_0^\theta|<\frac{k_0}{8}$ on $z=0$ leads to
\begin{equation}\label{u_bd2}
|u^{\theta}(r,0,t)| \leq \frac{1}{4}k_0 \quad\text{ for all }r\in[\pi,2\pi], t\geq 0.
\end{equation}

Defining $Q:=[\pi,2\pi] \times [0,\pi]$ to be a square on the $rz$-plane, the above symmetry results give $(u^r, u^z) \cdot n = 0$ on $\partial Q$ for all times.
Using this boundary condition as well as the divergence-free property of $(r u^r, r u^z)$ in $(r,z)$ (which follows from \eqref{bs_euler}), we apply the divergence theorem to obtain
\[ 
\begin{split}
\frac{d}{dt} \int_Q \omega^\theta(r,z,t) \, drdz &= \int_Q (ru_r, ru_z) \cdot \nabla_{r,z}\left(\frac{\omega^\theta}{r}\right) + \frac{\partial_z (u^\theta)^2}{r}\,drdz \\
&= \int_Q \frac{\partial_z (u^\theta)^2}{r}\,drdz \\
&= \int_\pi^{2\pi} \frac1r \left(u^\theta(r,\pi,t)^2 - u^\theta (r,0,t)^2 \right) dr \\
&\geq (\ln 2) \frac{3}{16}k_0^2  \geq  \frac{1}{10}k_0^2 
\end{split}
\]
for all times during the existence of a smooth solution, where the last inequality follows from \eqref{u_bd1} and \eqref{u_bd2}. This directly implies
\[
A(t) := \int_Q \omega^\theta(r,z,t) \, drdz \geq \frac{1}{10}k_0^2  t + \int_Q \omega^\theta_0 \, drdz.
\]
In particular, if $\int_Q \omega^\theta_0 drdz \geq 0$, this implies
\begin{equation}\label{a1}
A(t) \geq\frac{1}{10}k_0^2  t \quad\text{ for all }t\geq 0;
\end{equation}
and if $\int_Q \omega^\theta_0 drdz < 0$, we have
\begin{equation}\label{a2}
A(t)  \geq \frac{1}{20}k_0^2  t \quad\text{ for all }t\geq T_0 =: 20k_0^{-2} \left|\int_Q \omega^\theta_0 \, drdz\right|.
\end{equation}
Another ingredient we need is the energy conservation. It is well-known that the kinetic energy is conserved for 3D Euler equation, i.e. $\int_\Omega |u(x,t)|^2 \,dx = \int_\Omega |u_0|^2 \,dx$. Since $\Omega$ has an inner boundary with positive radius $\pi$, this implies in the domain $Q$ in the $rz$ plane, we also have
\[
\int_Q \left( u^r(r,z,t)^2 + u^z(r,z,t)^2\right) drdz \leq E_0(u_0).
\]
Recall that $\omega^\theta$ and $(u^r,u^z)$ are related by $\omega^\theta = \partial_r u^z - \partial_z u^r$. Thus we can apply Lemma~\ref{omega_lp} to conclude that
\[
\|\omega^\theta(t)\|_{L^p(Q)}\geq c_0  E_0^{-1+\frac{1}{p}}|A(t)|^{3-\frac{2}{p}} \quad\text{ for all }p\in[1,\infty], t\geq 0.
\]
which directly leads to \eqref{omlbiE} once we plug the estimates  \eqref{a1} and \eqref{a2} of $A(t)$ into the above equation.

Finally, applying the Green's theorem in $Q$, we have
\[  A(t) =\int_Q \omega^\theta dr dz = \int_Q (\partial_r u^z - \partial_z u^r)\, dr dz = \int_{\partial Q} u \cdot dl \leq 4\pi\|u(t)\|_{L^\infty}. \]
Combining this with the estimates  \eqref{a1} and \eqref{a2} directly gives \eqref{euler_u_bd}, finishing the proof. 
\end{proof}

\begin{appendix}
\section{Proof of Lemma~\ref{smallinx1}}
In the appendix we prove Lemma~\ref{smallinx1}. The proof is almost the same as in \cite{KY} other than a small improvement in part (a). We sketch a proof for both parts below for the sake of completeness.

\begin{proof}[\textup{\textbf{Proof of Lemma~\ref{smallinx1}, part (a)}}]
Here the proof mostly follows from \cite[equation (3.4)]{KY}, except that we make a small improvement dropping the assumption $\|\partial_1 \mu\|_{\dot{H}^{-1}}^2<\frac{1}{4}\|\mu\|_{L^2}^2$ in \cite{KY}.
Let us define
\[
\delta := \|\partial_{1} \mu\|_{\dot{H}^{-1}(\R^2)}^2,\quad A  := \|\mu\|_{L^2(\R^2)}^2. \]
Clearly, $\delta = \int_{\mathbb{R}^2} \frac{\xi_1^2}{|\xi|^2} |\hat \mu|^2 d\xi \leq A$. Let us discuss the following two cases. 

\noindent \emph{Case 1.} $\delta < \frac{A}{4}$. In this case let us define $
D_\delta:= \left\{(\xi_1,\xi_2): \frac{|\xi_1|}{|\xi|} \geq \sqrt{\frac{2\delta}{A}}\right\}.
$
By definition of $D_\delta$, we have 
\[
 \delta \geq \int_{D_\delta}  \frac{\xi_1^2}{|\xi|^2} |\hat\mu(\xi)|^2 d\xi \geq \frac{2\delta}{A} \int_{D_\delta} |\hat \mu|^2 d\xi.
\]
This gives $\int_{D_\delta} |\hat \mu|^2 d\xi \leq \frac{1}{2}A$, thus $\int_{D_\delta^c} |\hat \mu|^2 d\xi \geq \frac{1}{2}A$. Note that $D_\delta^c$ can be expressed in polar coordinates as $D_\delta^c = \{(r \cos\theta,r\sin\theta): r\geq 0, |\cos\theta| < \sqrt{2\delta/A}$\}.

Since $\mu\in C_c^\infty(\mathbb{R}^2)$, we have $\|\hat\mu\|_{L^\infty(\R^2)} \leq (2\pi)^{-1}\|\mu\|_{L^1(\R^2)} =: B$. Let $h_\delta>0$ be such that $|D_\delta^c \cap \{|\xi_2|<h_\delta\}| = (4B^2)^{-1} A$, which we will estimate later. Such definition gives
\[
\int_{D_\delta^c \cap \{|\xi_2|\geq h_\delta\}} |\hat\mu|^2 d\xi = \int_{D_\delta^c}  |\hat\mu|^2 d\xi - \int_{D_\delta^c \cap \{|\xi_2|<h_\delta\}}  |\hat\mu|^2  d\xi \geq \frac{1}{2}A  - (4B^2)^{-1} A B^2= \frac{1}{4}A,
\]
which implies
\begin{equation}\label{Hs}
\|\mu\|_{\dot{H}^s(\R^2)}^2 \geq \int_{\mathbb{R}^2} |\xi_2 |^{2s} |\hat\mu|^2 d\xi \geq h_\delta^{2s} \int_{D_\delta^c \cap \{|\xi_2|\geq h_\delta\}} |\hat\mu|^2 d\xi \geq \frac{A}{4} h_\delta^{2s}.
\end{equation}
To estimate $h_\delta$, let us denote $\theta_0 := \cos^{-1}(\sqrt{\frac{2\delta}{A}})$. Since $D_\delta^c \cap \{|\xi_2|<h_\delta\}$ consists of two identical triangles with height $h_\delta$ and base $2h_\delta \cot\theta_0$, we have
\[
 (4B^2)^{-1} A = |D_\delta^c \cap \{|\xi_2|<h_\delta\}| = 2h_\delta^2 \cot\theta_0 \leq 4\sqrt{\delta} A^{-1/2}h_\delta^2,
\]
where the inequality follows from  $\cos\theta_0 = \sqrt{\frac{2\delta}{A}}$ and $\sin\theta_0 = \sqrt{1-\frac{2\delta}{A}} \geq 1/\sqrt{2}$, due the assumption $\delta < A/4$ in case 1.
Therefore $h_\delta \geq (4B)^{-1} A^{3/4} \delta^{-1/4}$. Plugging it into \eqref{Hs} yields
\[
\|\mu\|_{\dot{H}^s(\R^2)} \geq  \frac{\sqrt{A}}{2} h_\delta^{s} \geq c(s,A,B) \delta^{-s/4}, 
\]
finishing the proof of \eqref{smallinx1} in case 1.

\noindent\emph{Case 2.} $\delta \geq \frac{A}{4}$. As in case 1, let us define $\|\hat\mu\|_{L^\infty(\R^2)} \leq (2\pi)^{-1}\|\mu\|_{L^1(\R^2)} =: B$. Let $r_0 := (A/(2\pi B^2))^{1/2}$. Such definition leads to
\[
\int_{B(0,r_0)} |\hat \mu|^2 d\xi \leq \pi r_0^2 \|\hat\mu\|_{L^\infty(\R^2)}^2  \leq \frac{A}{2},
\]
thus
\[
\|\mu\|_{\dot{H}^s(\R^2)}^2 \geq \int_{B(0,r_0)^c} |\xi |^{2s}  |\hat \mu|^2 d\xi \geq r_0^{2s} \frac{A}{2} \geq c(s,A,B) \delta^{-s/4},
\]
where the last inequality follows from the assumption $\delta \geq \frac{A}{4}$ in case 2. This finishes the proof of part (a).\end{proof}

\begin{proof}[\textup{\textbf{Proof of Lemma~\ref{smallinx1}, part (b)}}]
 This part is equivalent with the last (unnumbered) equation in the proof of Theorem 1.2 in \cite{KY}. We sketch a proof below for completeness, and also to clarify the dependence of $c_2(s, \int_{\T\times[0,\pi]} \mu^{1/3} dx)$ in \eqref{eq_delta2}.
 
  For any $k=(k_1,k_2)\in \Z^2$, the Fourier coefficient $\hat\mu(k_1,k_2)$  can be written as
\begin{equation}\label{rho_g}
\begin{split}
\hat\mu(k_1,k_2)
 &= \frac{1}{(2\pi)^2} \int_{\T} e^{-ik_1 x_1} \int_{\T} e^{-ik_2 x_2} \mu(x_1, x_2) dx_2 dx_1\\
 &= \frac{1}{(2\pi)^2}  \int_{\T} e^{-ik_1 x_1} (-2i) \underbrace{\int_{0}^\pi \sin(k_2 x_2) \mu(x_1, x_2) dx_2}_{=:g(x_1,k_2)} dx_1,
\end{split}
\end{equation}
where the last identity is due to $\mu$ being odd in $x_2$. With $g(x_1,k_2)$ defined in the last line of \eqref{rho_g}, when setting $k_2=1$, we claim that $g(x_1,1)$ satisfies the following properties:
\begin{enumerate}
\item[(a)] $g(x_1,1)$ is even in $x_1$ and nonnegative for all $x_1\in\mathbb{T}$.
\item[(b)] $g(0,1)=0$.
\item[(c)] $\int_{\T} g(x_1, 1) dx_1 \geq c(\int_D \mu(x)^{1/3} dx)^3$ for some universal constant $c>0$.
\end{enumerate}
Here property (a) follows from the facts that $\mu$ is even in $x_1$ and nonnegative on $D:= [0,\pi]^2$. Property (b) follows from $\mu(0,\cdot)\equiv 0$. For property (c), note that
\[
\int_{\T} g(x_1,1) dx_1 = 2 \int_0^\pi g(x_1,1) dx_1 = 2\int_D \sin(x_2) \mu(x) dx.
\]
Combining H\"older's inequality with the fact that $\sin(x_2)\mu(x)\geq 0$ in $D$, we have
\[
\int_D \sin(x_2) \mu(x) dx \geq  \left(\int_D \sin(x_2)^{-1/2} dx\right)^{-2} \left(\int_D \mu(x)^{1/3} dx\right)^3 \geq c_0 \left(\int_D \mu(x)^{1/3} dx\right)^3,
\]
for some universal constant $c_0>0$. This proves property (c).

For any $k_1\in\mathbb{Z}$, let $\hat g(k_1)$ be the Fourier coefficient of $g(\cdot, 1)$, that is,
\begin{equation}\label{rho_g_2}
\hat g(k_1) :=\frac{1}{2\pi} \int_{\T} e^{-ik_1 x_1} g(x_1,1) dx.
\end{equation}
Denote by $
\bar g := \frac{1}{2\pi}\int_{\T} g(x_1,1) dx_1
$ the average of $g(\cdot,1)$. Applying the definition of $\hat g$ to \eqref{rho_g} gives
\begin{equation}
\label{eq_k_rho}
 \hat\mu(k_1,1)=\frac{-2i}{2\pi}\hat g(k_1)\quad\text{ for any }k_1\in \Z.
 \end{equation} This allows us to bound $\delta := \|\partial_{1}\mu\|_{\dot{H}^{-1}(\T^2)}^2$ from below as
\begin{equation}
\begin{split}
\delta \geq (2\pi)^2\sum_{k_1\in \Z\setminus\{0\}} \frac{k_1^2}{k_1^2+1}|\hat\mu(k_1,1)|^2
&\geq 2\sum_{k_1\in \Z\setminus\{0\}} |\hat g(k_1)|^2 = \frac{1}{\pi} \int_{\T} |g(x_1,1)-\bar g|^2 dx_1.
\end{split}
\end{equation}

By property (c), $\bar g\geq c(\int_D \mu(x)^{1/3} dx)^3>0$.
Applying \cite[Lemma~3.3]{KY}  to $g(x_1,1)$ yields
\begin{equation}\label{ineq_g}
\|g(\cdot,1)\|_{\dot{H}^{s}(\T)} \geq c\left(s, \int_D \mu(x)^{1/3} dx\right) \delta^{-s+\frac{1}{2}} \quad\text{ for all }s>\frac{1}{2}.
\end{equation}
Note that
\begin{equation}\label{ineq_g2}
\|g(\cdot,1)\|_{\dot{H}^s(\T)}^2  
= 2\pi^3 \sum_{k_1\neq 0} |k_1|^{2s}|\hat\mu(k_1,1)|^2 \leq\frac{\pi}{\sqrt{2}}  \|\partial_{1}\mu\|_{\dot{H}^{s-1}(\T^2)}^2 \leq \frac{\pi}{\sqrt{2}}  \|\mu\|_{\dot{H}^{s}(\T^2)}^2,
\end{equation}
where the first inequality follows by the assumption $s>1/2$. Finally, combining \eqref{ineq_g} and \eqref{ineq_g2} gives \eqref{eq_delta2}. 
 \end{proof}

\end{appendix}
\bibliographystyle{abbrv}
\bibliography{references}

\end{document}